\documentclass[13pt, reqno]{amsart}
\usepackage{dsfont}
\usepackage{amsmath}
\usepackage{amscd}
\usepackage{amsthm}
\usepackage{amssymb} \usepackage{latexsym}
\usepackage{eufrak}
\usepackage{euscript}
\usepackage{epsfig}
\usepackage{graphics}
\usepackage{array}
\usepackage{enumerate}
\usepackage{color}
\usepackage{wasysym}
\usepackage{pdfsync}
\usepackage{stmaryrd}
\usepackage{url}
\usepackage[font=small, labelformat=simple]{caption}

\newcommand{\bel}[1]{\begin{equation*}\label{#1}}
	
	\newcommand{\be}{\begin{equation}}

		\newcommand{\ba}{\begin{eqnarray}}
			\newcommand{\ea}{\end{eqnarray}}

		\newcommand{\qe}{\end{equation}}
	\newcommand{\R}{{\mathbb R}}
	\newcommand{\N}{{\mathbb N}}
	\newcommand{\Z}{{\mathbb Z}}
	\newcommand{\C}{{\mathbb C}}

	\newcommand{\eg}{\begin{example}}
		\newcommand{\egd}{\end{example}}
	\newcommand{\tm}{\begin{thm}}
		\newcommand{\tmd}{\end{thm}}
	\newcommand{\co}{\begin{coro}}
		\newcommand{\cod}{\end{coro}}
	\newcommand{\enu}{\begin{enumerate}}
		\newcommand{\enud}{\end{enumerate}}
	\newcommand{\rmk}{\begin{rem}}
		\newcommand{\rmkd}{\end{rem}}
	
	\theoremstyle{theorem}
	\newtheorem{thm}{Theorem}[section]
	\newtheorem{prop}[thm]{Proposition}
	\theoremstyle{example}
	\newtheorem{example}[thm]{Example}
	\newtheorem{coro}[thm]{Corollary}
	\theoremstyle{lemma}
	\newtheorem{lemma}[thm]{Lemma}
	\theoremstyle{definition}
	\newtheorem{defi}[thm]{Definition}
	\theoremstyle{proof}
	
	\theoremstyle{remark}
	\newtheorem{rem}[thm]{Remark}
	\theoremstyle{remark}

	\UseRawInputEncoding

\begin{document}

		\title[The well-posedness and scattering theory of nonlinear Schr\"{o}dinger equations on lattice graphs]{The well-posedness and scattering theory of nonlinear Schr\"{o}dinger equations on lattice graphs }

		\author{Jiajun Wang}
		\address{Jiajun Wang: School of Mathematical Sciences,
			Fudan University, Shanghai 200433, China.}
		\email{21300180146@m.fudan.edu.cn}
		
		\begin{abstract}
		In this paper, we introduce a novel first-order derivative for functions on a lattice graph, which extends the discrete Laplacian and generalizes the theory of discrete PDEs on lattices. First, we establish the well-posedness of generalized discrete quasilinear Schr\"{o}dinger equations and give a new proof of the global well-posedness of discrete semilinear Schr\"{o}dinger equations. Then we provide explicit expressions of higher-order derivatives of the solution map and prove the analytic dependence between the solution and the initial data. At the end, we show the existence of the wave operator and prove the asymptotic completeness of the solutions with the small data.
		\end{abstract}

		\maketitle
		\numberwithin{equation}{section}
		\section{Introduction}
		The discrete nonlinear Schr\"{o}dinger equation (DNLS) is an important model. In physics, the DNLS has strong connections with amorphous materials \cite{11,12} and discrete self-trapped beams \cite{13,14}. In mathematics, the importance of the DNLS is reflected in two aspects. The first aspect is that it's a direct analogue of the nonlinear Schr\"{o}dinger equation (NLS), which possesses a substantial amount of interesting theories, like the scattering and soliton theories\cite{17,18,19,15,16}, and has a wide range of applications in describing water waves \cite{20,21} and polarized waves in optical ﬁbers \cite{22,23}. Another aspect is that an increasing number of mathematicians are engaged in the study of the DNLS. For example, R.Grande studies the continuum limit of discrete fractional semilinear Schr\"{o}dinger equations in 1 dimension \cite{24}. B.Choi and A.Aceves extend it to 2 dimension, with sharp dispersive
		estimates \cite{25}.
		\par 
		In this paper, we first introduce the following novel first-order derivative, which we call the discrete partial derivative. 
		\begin{defi}
			The discrete partial derivative $\partial_{j}u$ for $u\in \bigcup_{0<p\le \infty}\ell^{p}(\Z^{d})$  is defined by the convolution operator $u\ast \varphi_{j}$, where $\varphi_{j}$ is given by 
			\begin{equation*}
				\varphi_{j}(m) :=\left\{
				\begin{array}{ll}
					\frac{-4i}{\pi (4a^{2}-1)}, \ & \mathrm{if}\  m=a e_j, a\in \Z,\\
					0,\ &\mathrm{otherwise,}
				\end{array}\right.
			\end{equation*}
			where $\lbrace e_{j} \rbrace_{j=1}^{d}$ is the standard coordinate basis of $\Z^{d}$.
		\end{defi}
		For the motivation of this definition, we shall briefly recall the discrete Fourier transform $\mathcal{F}$ and its inverse $\mathcal{F}^{-1}$. 
		\begin{defi}
			For $u\in \ell^1(\Z^d)$ and $g\in L^{1}(\mathbb{T}^{d})$, the discrete Fourier transform $\mathcal{F}$ and the inverse discrete Fourier transform $\mathcal{F}^{-1}$ are defined as
			\begin{equation*}
				\mathcal{F}(u)(x):=\sum_{k\in \Z^{d}}u(k)e^{-ikx}, \quad \forall x\in \mathbb{T}^{d},
			\end{equation*}
			\begin{equation*}
				\mathcal{F}^{-1}(g)(k):=\frac{1}{(2\pi)^{d}}\int_{\mathbb{T}^d} g(x)e^{ikx}dx, \quad \forall k\in \Z^{d},
			\end{equation*}
			where $\mathbb{T}^{d}$ is the $d$-dimensional torus parametrized by $[0,2\pi)^d$. Besides, the discrete Fourier transform $\mathcal{F}$ can be extended to an isometric isomorphism between $\ell^{2}(\Z^{d})$ and $L^{2}(\mathbb{T}^{d})$.
		\end{defi} 
		Then we point out the connection between the discrete partial derivative $\partial_{j}$ and the discrete Laplacian $\Delta$, see Proposition 2.1 in \cite{26}.
		\begin{prop}
			For any $u\in C_0(\Z^d)$,
			\begin{equation*}
				\partial_{j}u= \mathcal{F}^{-1}\left(2i\cdot \sin(\frac{x_{j}}{2})\mathcal{F}(u)(x)\right),\quad     j=1,\cdots ,d,
			\end{equation*} 
			where $x=(x_{1},\cdots, x_{d})$, $C_0(\Z^d)$ denotes the set of finitely supported functions on $\Z^{d}$. In particular, we have $\Delta=\sum_{j=1}^{d}\partial_{j}\circ\partial_{j}$.
		\end{prop}
		\begin{rem}
			We recall that the standard difference operator $D_{j}$ is defined as 
			\begin{equation*}
				D_j u(m)=u(m+e_j)-u(m),\quad u:\Z^d\to \C,\; m\in \Z^d.
			\end{equation*}
			However, direct calculation shows that $\Delta\ne \sum_{j=1}^{d} D_{j}\circ D_{j}$, which is one of the major reasons for introducing the discrete partial derivative $\partial_{j}$. Besides, in some sense, the standard difference operator $D_{j}$ and the discrete partial derivative $\partial_{j}$ are equivalent. To be more concrete, we have
			\begin{equation*}
				\|\partial_{j}u\|_{\ell^{p}(\Z^{d})}\approx \|D_{j}u\|_{\ell^{p}(\Z^{d})},\quad	\|\partial u\|_{\ell^{p}(\Z^{d})}\approx \|Du\|_{\ell^{p}(\Z^{d})}, \quad 1<p<\infty,
			\end{equation*}
			where $\partial=(\partial_{1},\partial_{2},\cdots,\partial_{d})$, $D=(D_{1},D_{2},\cdots,D_{d})$. For more details, we refer to Theorem 2.6 in \cite{26} or \cite{27,28}.
		\end{rem}
		\begin{rem}
			The discrete partial derivative $\partial_{j}$ has a direct connection with the famous fractional discrete Laplacian $(-\Delta)^{\frac{\alpha}{2}}$, which is extensively studied, for example, in \cite{25,29}.
		\end{rem}
		Before we present our main results, we shall briefly introduce the definition of the weighted $\ell^{p}(\Z^{d})$ space.
		\begin{defi}
			For $0<p\le \infty$, $\alpha\in \R$, 
			the $\ell^{p,\alpha}$ norm is defined as
			$$\|f\|_{\ell^{p,\alpha}(\Z^{d})}:=\|f_\alpha\|_{\ell^p(\Z^d)},$$ where $f_\alpha(m):=f(m)\langle m \rangle^\alpha, m\in\Z^d.$ The definition of $\langle m \rangle$ refers to the end of this introduction. We write
			\begin{equation*}
				\ell^{p,\alpha}(\Z^{d}):=\lbrace f:\Z^d\to \C; \|f\|_{\ell^{p,\alpha}(\Z^{d})}<\infty \rbrace.
			\end{equation*}
		\end{defi}
		\begin{rem}\label{ee}
			Notice that, $\ell^{2,s}(\Z^{d})$ space has an isomorphism with the Sobolev $H^{s}(\mathbb{T}^{d})$ space, when $s\in \N$, see Theorem 2.10 in  \cite{26}. 
		\end{rem}
		Now, in the regime of $\ell^{2,s}(\Z^{d})$ space, we prove the local well-posedness of  generalized discrete quasilinear Schr\"{o}dinger equations and give a new proof of the global well-posedness of discrete semilinear  Schr\"{o}dinger equations.
		\begin{thm}\label{1}
			For the following generalized discrete quasilinear Schr\"{o}dinger equation
			\begin{equation}\label{nonlinear}
				\left\{
				\begin{aligned}
					& i\partial_{t} u(x,t)+ g^{jk}(u,\partial u)\partial_{jk} u(x,t) =F(u,\partial u),  \\
					& u(x,0) = u_{0}(x), \quad (x,t)\in\Z^d\times \R,
				\end{aligned}
				\right.
			\end{equation}
			where we follow the Einstein's summation convention, with $\partial_{jk}=\partial_{j}\circ\partial_{k}$. Then if $g^{jk}\in C^{1}(\mathbb{R}\times \mathbb{R}^{d})$, $F\in C^{1}(\mathbb{R}\times \mathbb{R}^{d})$, $F(0,0)=0$ and the initial data $u_{0}\in \ell^{2,s}(\Z^{d})$, $s\in [0,1]$, then there exists $T>0$ and the equation (\ref{nonlinear}) has a unique classical solution $u\in C^{1}([0,T];\ell^{2,s}(\Z^{d}))$. Furthermore, we have the continuation criterion that if the maximal existence time $T^{\ast}$ is finite, then $\|u(\cdot,t)\|_{\ell^{\infty}(\Z^{d})}$ is unbounded in $\left[0,T^{\ast}\right)$.
		\end{thm}
		\begin{thm}\label{2}
			For the following discrete semilinear  Schr\"{o}dinger equation
			\begin{equation}\label{NLS}
				\left\{
				\begin{aligned}
					& i\partial_{t} u(x,t)+ \dfrac{1}{2}\Delta u(x,t) =\mu|u|^{p-1}u,\;  \mu=\pm 1,\; p>1,\\
					& u(x,0) = u_{0}(x), \quad (x,t)\in\Z^d\times \R,
				\end{aligned}
				\right.
			\end{equation}
			where the case $\mu=-1$ is called the focusing case and $\mu=1$ is called the defocusing case, if $u_{0}\in \ell^{2,s}(\Z^{d})$, $s\in [0,1]$, then we have a unique classical solution $u\in C^{1}(\left[0,+\infty\right);\ell^{2,s}(\Z^{d}))$.
		\end{thm}
		Besides, we also give an explicit expression of the higher order derivatives of the solution map.
		\begin{thm}\label{3}
			For any $T>0$, the map  $\Phi:\ell^{2,s}(\Z^{d}) \to C([0,T];\ell^{2,s}(\Z^{d}))$, $s\in[0,1]$, is defined as the solution map of the equation (\ref{NLS}), i.e. for any  $u_{0}\in \ell^{2,s}(\Z^{d})$, $\Phi(u_{0})$ satisfies 
			\begin{equation*}
				\left\{
				\begin{aligned}
					& i\partial_{t} \Phi(u_{0})(x,t)+ \dfrac{1}{2}\Delta \Phi(u_{0})(x,t) =\mu|\Phi(u_{0})|^{p-1}\Phi(u_{0}), \;  \mu=\pm 1,\; p>1, \\
					& \Phi(u_{0})(x,0) = u_{0}(x), \quad (x,t)\in\Z^d\times \R.
				\end{aligned}
				\right.
			\end{equation*}
			If $p$ is a positive odd integer, then $\Psi^{k}$, see Definition \ref{higher order}, is the $k^{th}$-order Fr\'{e}chet derivative of $\Phi$ and, for any fixed $u_{0}\in \ell^{2,s}(\Z^{d})$, $\Psi_{u_{0}}^{k}=\Psi^{k}(u_{0})$ is a bounded $k$-linear map from $\underbrace{\ell^{2,s}(\Z^{d})\times\cdots \times \ell^{2,s}(\Z^{d})}_{k\; times}$ to $C([0,T];\ell^{2,s}(\Z^{d}))$.
		\end{thm}
		In a special form, we can even prove the analytic dependence on parameter $\varepsilon$.
		\begin{thm}\label{4}
			Let $p$ be a positive odd integer and $u^{(\varepsilon)}\in C^{1}([0,T];\ell^{2,s}(\Z^{d}))$, $s\in [0,1]$, $\forall \varepsilon\in \mathbb{R}$, be the classical solution of the equation (\ref{NLS}) with the initial data $\varepsilon u_{0}(x)$, i.e.
			\begin{equation}\label{p}
				\left\{
				\begin{aligned}
					& i\partial_{t} u^{(\varepsilon)}(x,t)+ \dfrac{1}{2}\Delta u^{(\varepsilon)}(x,t) =\mu|u^{(\varepsilon)}|^{p-1}u^{(\varepsilon)},\;  \mu=\pm 1,\; p>1,\\
					& u^{(\varepsilon)}(x,0) = \varepsilon u_{0}(x), \quad (x,t)\in\Z^d\times \R,
				\end{aligned}
				\right.
			\end{equation}
			where $u_{0}\in \ell^{2,s}(\Z^{d})$ is fixed. Then $u^{(\varepsilon)}$ is a real analytic function with respect to $\varepsilon$, which shows the solution will analytically depend on initial data.
		\end{thm}
		Before we state other results, we shall introduce the definitions of the scattering, the wave operator and the asymptotic completeness.
		\begin{defi}
			We say a global classical solution $u\in C^{1}(\left[0,+\infty\right);\ell^{2}(\Z^{d}))$ of discrete semilinear Sch\"{o}dinger equation (\ref{NLS}) with the initial data $u(x,0)=u_{0}(x)\in \ell^{2}(\Z^{d})$ scatters in $\ell^{2}(\Z^{d})$ to a solution of free discrete Sch\"{o}dinger equation with the initial data $u_{+}\in \ell^{2}(\Z^{d})$ as $t\to +\infty$ if 
			\begin{equation*}
				\|u(\cdot,t)-e^{it\Delta/2}u_{+}\|_{\ell^{2}(\Z^{d})} \to 0, \quad t\to \infty,
			\end{equation*}
			or equivalently 
			\begin{equation*}
				\|e^{-it\Delta/2}u(\cdot,t)-u_{+}\|_{\ell^{2}(\Z^{d})} \to 0, \quad t\to \infty.
			\end{equation*}
		\end{defi}
		\begin{defi}
			If for every asymptotic state $u_{+}\in \ell^{2}(\Z^{d})$, there exists a unique initial data $u_{0}\in \ell^{2}(\Z^{d})$, whose corresponding global classical solution $u\in C^{1}(\left[0,+\infty\right);\ell^{2}(\Z^{d}))$ scatters to $e^{it\Delta/2}u_{+}$ as $t\to +\infty$. Then the wave operator $\Omega_{+}:\ell^{2}(\Z^{d})\to \ell^{2}(\Z^{d})$ is defined by $\Omega_{+}(u_{+}):=u_{0}$.
		\end{defi}
		\begin{defi}
			If for a classical solution $u\in C^{1}(\left[0,+\infty\right);\ell^{2}(\Z^{d}))$ of discrete semilinear Sch\"{o}dinger equation (\ref{NLS}), there exists $u_{+}\in \ell^{2}(\Z^{d})$ and $u$ scatters to $e^{it\Delta/2}u_{+}$, then we say that we have the asymptotic completeness of $u$.
		\end{defi}
		Now, we state our results about the existence of the wave operator and the asymptotic completeness.
		\begin{thm}\label{5}
			If $6\le (p-1)d$, then the wave operator $\Omega_{+}:\ell^{2}(\Z^{d})\to \ell^{2}(\Z^{d})$ exists and it's continuous in the $\ell^{2}(\Z^{d})$-norm.
		\end{thm}
		\begin{thm}\label{6}
			If $6\le (p-1)d$, then there exists $\varepsilon=\varepsilon(d)>0$, such that whenever $\|u_{0}\|_{\ell^{2}(\Z^{d})}\le \varepsilon$, we have the asymptotic completeness of $u$.
		\end{thm}
		As an extension of the above well-posedness and the scattering results, we also studied the long-time well-posedness theory and the soliton solution or the ground state of the DNLS.
		\begin{thm}\label{7}
			For the following generalized discrete quasilinear Schr\"{o}dinger equation with the small initial data
			 \begin{equation}\label{small data}
			 	\left\{
			 	\begin{aligned}
			 		& i\partial_{t} u(x,t)+ g^{jk}(u,\partial u)\partial_{jk} u(x,t)= F(u,\partial u), \\
			 		& u(x,0) = \varepsilon u_{0}(x), \quad (x,t)\in\Z^d\times \R,
			 	\end{aligned}
			 	\right.
			 \end{equation}
			 if $s\in[0,1]$,  $u_{0}\in \ell^{2,s}(\Z^{d})$, then there exist $\delta>0$ and constant $K=K(g^{jk},F,u_{0},d)>0$ such that $ \forall 0<\varepsilon<\delta$, the maximal existence time $T^{\ast}\ge K\log(\log(\frac{1}{\varepsilon}))$.
	    \end{thm}
	    \begin{thm}\label{8}
	    	For the following ground state equation, derived from the soliton solution of the focusing DNLS,  
	    	\begin{equation}\label{ground}
	    		-\Delta Q(x)+2\tau Q(x)=2|Q|^{p-1}Q,\; x\in \Z^{d}, \; \tau>0,
	    	\end{equation}
	    	the three tables below illustrate the uniformly lower bound property of the ground state $Q$, where, for simplicity, we denote $\sigma:=\frac{p-1}{2}$, $q:=\frac{2d(2\sigma+1)}{d(2\sigma+1)+3}$, $\sigma_{\ast}$ as the positive root of $\frac{\sigma(2\sigma+1)}{\sigma+1}=\frac{3}{d}$.
	    	\begin{table}[htb]   
	    		\captionsetup{skip=3pt} 
	    		\caption{\Large{$d=1$}}  
	    		\label{table:1} 
	    		\begin{center} 
	    			\begin{tabular}{|m{1.8cm}<{\centering}|m{2.8cm}<{\centering}|m{2.9cm}<{\centering}|m{1.8cm}<{\centering}|m{2.9cm}<{\centering}|}   
	    				\hline   \textbf{} & \textbf{$0<\sigma<\sigma_{\ast}$} & \textbf{$\sigma_{\ast}\le \sigma< \frac{2}{d}$} & \textbf{$\frac{2}{d}\le \sigma <\frac{3}{d}$} & \textbf{$\sigma\ge \frac{3}{d}$}\\   
	    				\hline   $\|\cdot\|_{\ell^{2}(\Z^{d})}$ & $\forall v >0, \exists \|Q\|=v$ & $\forall v >0, \exists \|Q\|=v$  & $\times$ & $\exists\varepsilon(d),\|Q\|\ge \varepsilon>0$\\ 
	    				\hline $\|\cdot\|_{\ell^{q}(\Z^{d})}$   & $\times$ & $\exists\varepsilon(d),\|Q\|\ge \varepsilon>0$  & $\times$ & $\exists\varepsilon(d),\|Q\|\ge \varepsilon>0$\\  
	    				\hline   
	    			\end{tabular} 
	    		\end{center}   
	    	\end{table}
	    		\begin{table}[htb]   
	    		\captionsetup{skip=3pt} 
	    		\caption{\Large{$d=2$}}  
	    		\label{table:2} 
	    		\begin{center} 
	    			\begin{tabular}{|m{1.8cm}<{\centering}|m{2.8cm}<{\centering}|m{2.9cm}<{\centering}|m{2.9cm}<{\centering}|}   
	    				\hline   \textbf{} & \textbf{$0<\sigma<\sigma_{\ast}=\frac{2}{d}$} & \textbf{$\sigma_{\ast}=\frac{2}{d}\le \sigma< \frac{3}{d}$}  & \textbf{$\sigma\ge \frac{3}{d}$}\\   
	    				\hline   $\|\cdot\|_{\ell^{2}(\Z^{d})}$ & $\forall v >0, \exists \|Q\|=v$ & $\forall v >0, \exists \|Q\|=v$ & $\exists\varepsilon(d),\|Q\|\ge \varepsilon>0$\\ 
	    				\hline $\|\cdot\|_{\ell^{q}(\Z^{d})}$ 																						 & $\times$  & $\exists\varepsilon(d),\|Q\|\ge \varepsilon>0$ & $\exists\varepsilon(d),\|Q\|\ge \varepsilon>0$\\  
	    				\hline   
	    			\end{tabular} 
	    		\end{center}   
	    	\end{table}
	    	\begin{table}[htb]   
	    		\captionsetup{skip=3pt} 
	    		\caption{\Large{$d\ge 3$}}  
	    		\label{table:3} 
	    		\begin{center} 
	    			\begin{tabular}{|m{1.8cm}<{\centering}|m{2.8cm}<{\centering}|m{1.8cm}<{\centering}|m{2.9cm}<{\centering}|m{2.9cm}<{\centering}|}   
	    				\hline   \textbf{} & \textbf{$0<\sigma<\frac{2}{d}$} & \textbf{$\frac{2}{d}\le \sigma<\sigma_{\ast}$} & \textbf{$\sigma_{\ast}\le \sigma <\frac{3}{d}$} & \textbf{$\sigma\ge \frac{3}{d}$}\\   
	    				\hline   $\|\cdot\|_{\ell^{2}(\Z^{d})}$ & $\forall v >0, \exists \|Q\|=v$ & $\times$  & $\times$ & $\exists\varepsilon(d),\|Q\|\ge \varepsilon>0$\\ 
	    				\hline $\|\cdot\|_{\ell^{q}(\Z^{d})}$   & $\times$ & $\times$  & $\exists\varepsilon(d),\|Q\|\ge \varepsilon>0$ & $\exists\varepsilon(d),\|Q\|\ge \varepsilon>0$\\  
	    				\hline   
	    			\end{tabular} 
	    		\end{center}   
	    	\end{table}
	    \end{thm}   
		We organize the paper as follows. In Section 2, we establish the useful conservation law and the energy estimate, which play a central role in establishing the well-posedness theory of generalized DNLS. In Section 3, we give the detailed proofs of Theorem \ref{1} and Theorem \ref{2}. In Section 4, we present the explicit expression of $\Psi^{k}$ and give the proofs of Theorem \ref{3} and Theorem \ref{4}. In Section 5, we finish the proofs of Theorem \ref{5} and Theorem \ref{6}. In section 6, the proofs of Theorem \ref{7} and Theorem \ref{8} will be presented.

		\noindent
		\textbf{Notation.}
		\begin{itemize}
			\item By $u\in C^{k}([0,T]; B)( or \; L^{p}([0,T];B))$ for a Banach space $B,$ we mean $u$ is a $C^{k}(or \; L^{p})$ map from $[0,T]$ to $B;$ see page 301 in \cite{32}.
			\item By $u\in C_{0}(\Z^{d}\times [0,T])$, we mean $u$ has compact support on $\Z^{d}\times [0,T]$.
			\item By $A\lesssim B$ (resp. $A\approx B$), we mean there is a positive constant $C$, such that $A\le CB$ (resp. $C^{-1}B\le A \le C B$). If the constant $C$ depends on $p,$ then we write $A\lesssim_{p}B$ (resp. $A\approx_{p} B$).
			\item By $T\in \mathcal{B}(U,V)$ for normed spaces $U,V$, we mean $T:U\to V$ is a bounded linear map.
			\item Set $\langle m \rangle:=(1+|m|^{2})^{1/2}$ and $|m|:=(\sum_{j=1}^{d}|m_{j}|^{2})^{1/2}$ for $m=(m_{1},\cdots, m_{d})\in \Z^{d}$.
		\end{itemize}
		\section{The conservation law and energy estimate}
		In this section, we establish two conservation laws (the mass conservation \& the energy conservation) and the energy estimates of discrete Schr\"{o}dinger equations. For conservation law, the energy conservation follows directly from time translation invariance and the Noether theorem, see \cite{1,2}, but mass conservation may fail to simply correspond with any usual symmetries of nonrelativistic motions, like space-time translation and space rotation. In fluid mechanics, the mass conservation may correspond to the gauge invariance of some field, see \cite{3}. For energy estimate, we will establish it on a more general framework, and it will be rather useful in the deeper theories.	
		\par
		For $u\in C^1([0,T];\ell^{2}(\Z^{d}))$, we define the mass and the energy as follows
		\begin{equation}\label{ww}
			M_{disc}[t]:= \sum_{k\in \Z^{d}}|u(k,t)|^{2},
		\end{equation}
		\begin{equation}\label{w}
			E_{disc}[t]:= \sum_{k\in \Z^{d}}|\partial u(k,t)|^2 + \frac{2\mu}{p+1}|u(k,t)|^{p+1}.
		\end{equation}
		Since for any $v\in \ell^2(\Z^d)$, we have the following identity
		\begin{equation*}
			\sum_{k\in\Z^{d}} |\partial v(k)|^{2}=\sum_{k\in\Z^{d}}|D v(k)|^2,
		\end{equation*}
		the energy can be rewritten as 
		\begin{equation}
			E_{disc}[t]:= \sum_{k\in \Z^{d}}|Du(k,t)|^2 +\frac{2\mu}{p+1}|u(k,t)|^{p+1}.
		\end{equation}
		\begin{thm}\label{conservation}
			If $u_{0}\in \ell^{2}(\Z^{d})$ and $u\in C^1([0,T];\ell^{2}(\Z^{d}))$ is a classical solution to the equation (\ref{NLS}), then we have the following two conservations
			\begin{equation*}
				M_{disc}[t]\equiv M[0], \quad E_{disc}[t]\equiv E_{disc}[0].
			\end{equation*}
		\end{thm}
		\begin{proof}
			We refer to Theorem 1.1 in \cite{30} and Proposition 4.1 in \cite{31} for the two conservations.
		\end{proof}
		In the rest of this section, we first develop the energy estimate for the following discrete inhomogeneous Schr\"{o}dinger equation
		\begin{equation}\label{inhomogeneous}
			\left\{
			\begin{aligned}
				& i\partial_{t} u(x,t)+ \frac{1}{2}\Delta u(x,t) =F(x,t),  \\
				& u(x,0) = u_{0}(x), \quad (x,t)\in\Z^d\times \R.
			\end{aligned}
			\right.
		\end{equation}
		\begin{thm}\label{energy estimate1}
			If $u_{0}\in \ell^{2,s}(\Z^{d})$, $s\ge 0,$ and $u\in C^1([0,T];\ell^{2,s}(\Z^{d}))$ is a classical solution to the equation (\ref{inhomogeneous}), then we have the following estimate
			\begin{equation*}
				\|u(\cdot,t)\|_{\ell^{2,s}(\Z^{d})}\lesssim (1+t^{s})\left(\|u_{0}\|_{\ell^{2,s}(\Z^{d})}+\int_{0}^{t}\|F(\cdot,s)\|_{\ell^{2,s}(\Z^{d})} ds\right).
			\end{equation*}
		\end{thm}
		\begin{proof}
			Notice that, applying the fact that the discrete Fourier transform $\mathcal{F}$ is an isomorphism between $\ell^{2,s}(\Z^{d})$ and $H^{s}(\mathbb{T}^{d})$, $s\in \N$, see Remark \ref{ee}. We can immediately derive the estimate as follows
			\begin{equation*}
				\|e^{it\Delta/2}u_{0}\|_{\ell^{2,s}(\Z^{d})}\lesssim (1+t^{s})\|u_{0}\|_{\ell^{2,s}(\Z^{d})},
			\end{equation*}
			where the operator $e^{it\Delta/2}$ is defined as
			\begin{equation*}
				e^{it\Delta/2}u:=\mathcal{F}^{-1}\left(e^{itK(\xi)/2}\mathcal{F}(u)(\xi)\right), \quad \xi=(\xi_{1},\cdots,\xi_{d}) \in \mathbb{T}^{d},
			\end{equation*}
			with $K(\xi)=\sum_{j=1}^{d}(2-2\cos(\xi_{j}))$.
			Then the whole theorem is a direct consequence of the following Dunamel formula 
			\begin{equation*}
				u(x,t)=e^{it\Delta/2}u_{0}-i\mu \int_{0}^{t}e^{i(t-s)\Delta/2}F(x,s) ds,
			\end{equation*}
			and complex interpolation for fractional $s$.
		\end{proof}
		Next, we will establish the energy estimate, which comes with an additional exponential term, for a more generalized framework as follows 
		\begin{equation}\label{generalized}
			\left\{
			\begin{aligned}
				& i\partial_{t} u(x,t)+ g^{jk}(x,t)\partial_{jk} u(x,t) =F(x,t),  \\
				& u(x,0) = u_{0}(x), \quad (x,t)\in\Z^d\times \R.
			\end{aligned}
			\right.
		\end{equation}
		\begin{thm}\label{energy estimate2}
			For $s\in[0,1]$, if $g^{jk}\in L^{1}([0,T];\ell^{\infty}(\Z^{d}))$, $F\in L^{1}([0,T];\ell^{2,s}(\Z^{d}))$, $u_{0}\in \ell^{2,s}(\Z^{d})$ and $u\in C^1([0,T];\ell^{2,s}(\Z^{d}))$ is a classical solution to the equation (\ref{generalized}), then we have the following estimate with some constant $C>0$
			\begin{equation*}
				\|u(\cdot,t)\|_{\ell^{2,s}(\Z^{d})}\lesssim \left(\|u_{0}\|_{\ell^{2,s}(\Z^{d})}+\int_{0}^{t}\|F(\cdot,s)\|_{\ell^{2,s}(\Z^{d})}\right)\cdot \exp\left(C\int_{0}^{t}\sum_{j,k}\|g^{jk}(\cdot,s)\|_{\ell^{\infty}(\Z^{d})} ds\right).
			\end{equation*}
		\end{thm}
		\begin{proof}
			We only consider the case $s=0$, as the general case is similar. We still focus on the  total mass $M(t)$ of the solution $u$ as follows
			\begin{equation*}
				M(t)=\sum_{k\in \Z^{d}}|u(k,t)|^{2}.
			\end{equation*}
			Then we differentiate the total mass with respect to the time variable $t$ and substitute the equation (\ref{generalized})
			\begin{equation*}
				\frac{d M(t)}{dt}= i\sum_{k\in \Z^{d}}\left(g^{jk}(k,t)\partial_{jk}u(k,t)-F(k,t)\right)\cdot\overline{u(k,t)}-u(k,t)\cdot \overline{(g^{jk}(k,t)\partial_{jk}u(k,t)-F(k,t))}
			\end{equation*}
			\begin{equation*}
				\lesssim \|F(\cdot,t)\|_{\ell^{2}(\Z^{d})}M(t)^{\frac{1}{2}}+\sum_{j,k}\|g^{jk}(\cdot,t)\|_{\ell^{\infty}(\Z^{d})}M(t).
			\end{equation*}
			In the above inequality, we use the estimate $\|\partial_{jk}u\|_{\ell^{2,s}(\Z^{d})}\lesssim \|\partial_{j}u\|_{\ell^{2,s}(\Z^{d})}\lesssim \|u\|_{\ell^{2,s}(\Z^{d})}$ from Theorem 2.14 in \cite{26} and complex interpolation.
			Then the energy estimate can be directly derived from the Gronwall inequality. 
		\end{proof}
		\section{The existence of solution to generalized DNLS}
		As the uniqueness is a direct consequence of the previous section's energy estimates, in this section, we will focus on establishing the local existence of generalized DNLS (\ref{nonlinear}) and the global existence of discrete semilinear  Schr\"{o}dinger equation (\ref{NLS}) i.e. Theorem \ref{1} and Theorem \ref{2}.
		\par 
		Before we establish the local existence of generalized DNLS, we shall first establish the existence of a solution to generalized discrete inhomogeneous Schr\"{o}dinger equation (\ref{generalized}), which is a foundation of the essential iteration argument. To derive this existence, we need to introduce the concept of a weak solution.
		\begin{defi}
			$u$ is a weak solution to the equation (\ref{generalized}) with the vanishing initial data, if it satisfies the following identity for any $\phi\in C_{0}^{\infty}((0,T)\times \Z^{d})$ 
			\begin{equation*}
				\int_{0}^{T}\sum_{m\in \Z^{d}}F(m,t)\overline{\phi(m,t)}  dt=\int_{0}^{T}\sum_{m\in \Z^{d}}u(m,t)\overline{L\phi(m,t)}  dt,
			\end{equation*}
			where $L\phi(x,t):=i \partial_{t}\phi(x,t)+\partial_{jk}(\overline{g^{jk}(x,t)}\phi(x,t))$.
		\end{defi}
		\begin{thm}\label{foundation}
			For $s\in[0,1]$, if $g^{jk}\in C([0,T];\ell^{\infty}(\Z^{d}))$, $F\in C([0,T];\ell^{2,s}(\Z^{d}))$, $u_{0}\in\ell^{2,s}(\Z^{d})$, then the equation (\ref{generalized}) has a classical solution $u\in C^1([0,T];\ell^{2,s}(\Z^{d}))$.
		\end{thm}
		\begin{proof}
			We first establish the existence of a weak solution and then lift it to a classical solution. Notice that we can assume the initial data is vanishing, otherwise we can consider $\widetilde{u}=u-u_{0}$. Next, we claim that
			\begin{equation}\label{k}
				\|\phi(\cdot,t)\|_{\ell^{2,-s}(\Z^{d})}\lesssim_{T}\int_{0}^{T}\|L\phi(\cdot,\tau)\|_{\ell^{2,-s}(\Z^{d})} d\tau, \quad \forall \phi \in C_{0}^{\infty}((0,T)\times \Z^{d}), \forall t\in [0,T].
			\end{equation}
			It suffices to prove the following estimate for any $v\in C^{1}([0,T];\ell^{2,-s}(\Z^{d}))$
			\begin{equation*}
				\|v(\cdot,t)\|_{\ell^{2,-s}(\Z^{d})}\lesssim_{T,g^{jk}}\left(\|v(\cdot,0)\|_{\ell^{2,-s}(\Z^{d})}+\int_{0}^{T}\|Lv(\cdot,\tau)\|_{\ell^{2,-s}(\Z^{d})} d\tau\right).
			\end{equation*}
			The proof of this estimate is similar to the proof in Theorem \ref{energy estimate2}, with the consideration of total mass and differentiation with respect to $t$, so we omitted it here.
			\par 
			Then we define the following linear space $V$ and the linear functional $\ell_{F}$ on it 
			\begin{equation*}
				V:=L^{\ast}C_{0}^{\infty}(\Z^{d}\times (0,T))=\lbrace Lv; v\in C_{0}^{\infty}(\Z^{d}\times (0,T)) \rbrace,
			\end{equation*} 
			\begin{equation*}
				\ell_{F}(Lv):=\int_{0}^{T}\sum_{m\in\Z^{d}}F(m,t)\overline{v(m,t)}dt.
			\end{equation*}
			Based on the estimate (\ref{k}), we can use the Cauchy-Schwarz inequality to derive the following boundedness
			\begin{equation*}
				|\ell_{F}(Lv)|\le \int_{0}^{T}\|F(\cdot,t)\|_{\ell^{2,s}(\Z^{d})}\|v(\cdot,t)\|_{\ell^{2,-s}(\Z^{d})} dt\lesssim_{T,F,g^{jk}}\int_{0}^{T}\|Lv(\cdot,t)\|_{\ell^{2,-s}(\Z^{d})}dt.
			\end{equation*}
			Since the space $V$ can be naturally embedded in $L^{1}([0,T];\ell^{2,-s}(\Z^{d}))$, then from the Hahn-Banach Theorem, $\ell_{F}$ can be extended to the whole  $L^{1}([0,T];\ell^{2,-s}(\Z^{d}))$. Notice that the dual space of $L^{1}([0,T];\ell^{2,-s}(\Z^{d}))$ is $L^{\infty}([0,T];\ell^{2,s}(\Z^{d}))$, we can obtain the weak solution $u\in L^{\infty}([0,T];\ell^{2,s}(\Z^{d}))$ to the equation $(\ref{generalized})$. Considering the following equation in the weak derivative sense 
			\begin{equation}\label{lllk}
				i\partial_{t}u(x,t)+g^{jk}(x,t)u(x,t)=F(x,t),
			\end{equation} 
			we can immediately deduce that $u\in C([0,T];\ell^{2,s}(\Z^{d}))$. Applying above identity (\ref{lllk}) again, with the regularity assumptions on $g^{jk}, F$, we can lift the weak solution $u$ to the classical solution, which belongs to $C^{1}([0,T];\ell^{2,s}(\Z^{d}))$.
		\end{proof}
		Next, we will give a proof of Theorem \ref{1}.
		
		\begin{proof}[Proof of Theorem \ref{1}]
			We decompose the whole proof into 4 steps.

			\noindent
			\textbf{Step 1:}
			We first set $u_{-1}\equiv 0$ and construct the iteration as follows to approximate the exact solution
			\begin{equation}\label{iteration}
				\left\{
				\begin{aligned}
					& i\partial_{t} u_{m}(x,t)+ g^{jk}(u_{m-1},\partial u_{m-1})\partial_{jk} u_{m}(x,t) =F(u_{m-1},\partial u_{m-1}),  \\
					& u_{m}(x,0) = u_{0}(x), \quad (x,t)\in\Z^d\times \R.
				\end{aligned}
				\right.
			\end{equation}
			Based on the existence and the uniqueness theory in Theorem $\ref{foundation}$, the above iteration (\ref{iteration}) is well-defined and each $u_{m}$ is a classical solution. Similarly, we introduce the weighted total mass for each $u_{m}$ as follows
			\begin{equation*}
				M_{m}(t):=\|u_{m}(\cdot,t)\|_{\ell^{2,s}(\Z^{d})}.
			\end{equation*}
			We will show that there exist $T>0, M>0$ such that  $M_{m}(t)\le M$, $\forall m,\forall t\in [0,T]$.

			\noindent
			\textbf{Step 2:}
			We prove this statement by induction. Taking $M$ large enough that $\|u_{0}\|_{\ell^{2,s}(\Z^{d})},1 \ll M$, then we can ensure $M_{0}(t)\le M$, $\forall t \in [0,1]$ from Theorem \ref{energy estimate2}. Then we suppose that there exists $0<T<1$, such that for $m\le n-1, \forall t\in [0,T]$, we have $M_{m}(t)\le M$. Then we will prove the case $m=n$.
			\par 
			Using the energy estimate in Theorem \ref{energy estimate2},  we obtain the inequality
			\begin{equation}\label{z}
				M_{n}(t)\lesssim \left(M_{n}(0)+\int_{0}^{t}\|F(u_{n-1},\partial u_{n-1})\|_{\ell^{2,s}(\Z^{d})} d\tau\right)\cdot \exp\left(C\int_{0}^{t}\sum_{j,k}\|g^{jk}(u_{n-1},\partial u_{n-1})\|_{\ell^{\infty}(\Z^{d})}d\tau\right).
			\end{equation}
			From the assumptions on $F,g^{jk}$ and induction hypothesis, the following inequalities hold 
			\begin{equation*}
				|F(u_{n-1},\partial u_{n-1})|\lesssim_{M}|u_{n-1}|+|\partial u_{n-1}|, \quad \sum_{j,k}\|g^{jk}(u_{n-1},\partial u_{n-1})\|_{\ell^{\infty}(\Z^{d})}\lesssim_{M}1.
			\end{equation*}
			Then we insert the above inequalities into the energy estimate (\ref{z}) and derive
			\begin{equation*}
				M_{n}(t)\le C_{1}(M_{n}(0)+C_{2}Mt)\cdot e^{C_{3}t},
			\end{equation*}
			where $C_{1}$ is independent of $M$ and $C_{2}, C_{3}$ depend on $M$. Notice that $M_{n}(0)$ is independent of $n, T, M$, thus we can take $T\ll M,1$, which is also independent of $n$, such that $M_{n}(t)\le M$. Thus, we complete the induction and $M_{n}(t)\le M$, $\forall t\in[0,T],\forall n$.

			\noindent
			\textbf{Step 3:}
			Next, we will show that $\lbrace u_{m}\rbrace$ is a Cauchy sequence in $C^{1}([0,T];\ell^{2,s}(\Z^{d}))$. Notice that, from the iteration in (\ref{iteration}), it suffices to show $\lbrace u_{m}\rbrace$ is a Cauchy sequence in $C([0,T];\ell^{2,s}(\Z^{d}))$ and iteration gives the following equations
			\begin{equation*}
				i\partial_{t}(u_{m}-u_{m-1})+g^{jk}(u_{m-1},\partial u_{m})\cdot \partial_{jk}(u_{m}-u_{m-1})=(\ast),
			\end{equation*}
			\begin{equation*}
				(\ast)=F(u_{m-1},\partial u_{m-1})-F(u_{m-2},\partial u_{m-2})+\left[g^{jk}(u_{m-2},\partial u_{m-2})-g^{jk}(u_{m-1},\partial u_{m-1})\right]\cdot \partial_{jk}u_{m-1}.
			\end{equation*}
			Thus, from the uniform boundedness of $\|u_{m}(\cdot,t)\|_{\ell^{2,s}(\Z^{d})}$ and the assumptions on $g^{jk}, F$, we deduce that
			\begin{equation*}
				(\ast)=O_{M}(|u_{m-1}-u_{m-2}|+|\partial u_{m-1}-\partial u_{m-2}|).
			\end{equation*} 
			Then we define $C_{m}(t)$ as follows to describe the difference of each term in $\lbrace u_{m} \rbrace$
			\begin{equation*}
				C_{m}(t):=\|u_{m}(\cdot,t)-u_{m-1}(\cdot,t)\|_{\ell^{2,s}(\Z^{d})}.
			\end{equation*}
			From the above estimate, we can apply the energy estimate in Theorem \ref{energy estimate2} and obtain 
			\begin{equation}\label{x}
				C_{m}(t)\le C\int_{0}^{t}C_{m-1}(\tau)d\tau,
			\end{equation}
			where $C$ is a positive constant that only depends on $M,T$. Then we apply (\ref{x}) for $m$ times and get the following inequality
			\begin{equation*}
				C_{m}(t)\le C^{m}\int \int \cdots\int_{0\le \tau_{1}\le \tau_{2}\le\cdots\le \tau_{m}\le t}C_{0}(\tau_{1}) d\tau_{1}d\tau_{2}\cdot d\tau_{m}\lesssim \frac{(Ct)^{m}}{m!}.
			\end{equation*}
			Thus, $\lbrace u_{m}\rbrace$ is a Cauchy sequence in $C^{1}([0,T];\ell^{2,s}(\Z^{d}))$ and we denote its limit as $u$, then $u\in C^{1}([0,T];\ell^{2,s}(\Z^{d}))$ is a classical solution to the equation (\ref{nonlinear}).

			\noindent
			\textbf{Step 4:}
			In the last step, we will show the continuation criterion. From the above argument, we have seen that if $\|u(\cdot,t)\|_{\ell^{2,s}(\Z^{d})}$ is bounded in $\left[0,T^{\ast}\right)$, then we can extend the solution over $T^{\ast}$. Therefore, we just need to show the weaker assumption $\|u(\cdot,t)\|_{\ell^{\infty}(\Z^{d})}$ is bounded in $\left[0,T^{\ast}\right)$ can imply the stronger assumption  $\|u(\cdot,t)\|_{\ell^{2,s}(\Z^{d})}$ is bounded in $\left[0,T^{\ast}\right)$.
			\par 
			Letting $n\to \infty$ in the estimate (\ref{z}), we can obtain 
			\begin{equation*}
				M(t)\lesssim \left(M(0)+\int_{0}^{t}\|F(u,\partial u)\|_{\ell^{2,s}(\Z^{d})} ds\right)\cdot \exp\left(C\int_{0}^{t}\sum_{j,k}\|g^{jk}(u,\partial u)\|_{\ell^{\infty}(\Z^{d})}ds\right),
			\end{equation*}
			where $M(t)=\|u(\cdot,t)\|_{\ell^{2,s}(\Z^{d})}$. Notice that the following inequalities 
			\begin{equation*}
				|F(u,\partial u)|\lesssim_{M}|u|+|\partial u|, \quad \sum_{j,k}\|g^{jk}(u,\partial u)\|_{\ell^{\infty}(\Z^{d})}\lesssim_{M} 1,
			\end{equation*}
			only require $\|u(\cdot,t)\|_{\ell^{\infty}(\Z^{d})}$ is uniformly bounded in $t\in [0,T]$. Therefore, we can further derive the estimate as follows
			\begin{equation*}
				M(t)\lesssim_{M,T^{\ast}} \left(M(0)+\int_{0}^{t}M(\tau)d\tau\right).
			\end{equation*}
			Then the Gronwall inequality directly ensures the boundedness of $\|u(\cdot,t)\|_{\ell^{2,s}(\Z^{d})}$, that's prove the continuation criterion.
		\end{proof}
		As a direct corollary of the above Theorem \ref{1}, we can obtain the global existence of discrete semilinear Schr\"{o}dinger equation (\ref{NLS}).
		
		\begin{proof}[Proof of Theorem \ref{2}]
			Based on the continuation criterion in Theorem \ref{1}, it suffices to show $\|u(\cdot,t)\|_{\ell^{2}(\Z^{d})}$ is bounded. Notice that, we have the mass conservation in Theorem \ref{conservation}, which completes the whole proof.
		\end{proof}
		\begin{rem}\label{backward}
			Notice that, the following backward discrete semilinear Schr\"{o}dinger equation 
			\begin{equation}\label{v}
				\left\{
				\begin{aligned}
					& -i\partial_{t} u(x,t)+ \dfrac{1}{2}\Delta u(x,t) =\mu|u|^{p-1}u,\;  \mu=\pm 1, \; p>1,\\
					& u(x,0) = u_{0}(x), \quad (x,t)\in\Z^d\times \R,
				\end{aligned}
				\right.
			\end{equation}
			also has a global classical solution with the initial data $u_{0}\in \ell^{2,s}(\Z^{d})$. To be more concrete, the equation (\ref{v}) also has the following mass conservation 
			\begin{equation*}
				\|u(\cdot,t)\|_{\ell^{2}(\Z^{d})}\equiv\|u_{0}\|_{\ell^{2}(\Z^{d})},
			\end{equation*}
			which can be similarly derived as in Theorem \ref{conservation}. Then, based on the continuation criterion given by Theorem \ref{1}, we can immediately obtain the existence of a  global solution to the backward discrete semilinear Schr\"{o}dinger equation (\ref{v}). Furthermore, the global well-posedness result can be generalized to the following equation
			\begin{equation*}
				\left\{
				\begin{aligned}
					& i\partial_{t} u(x,t)+\alpha\Delta u(x,t) =\beta F(|u|^{2})u, \; \alpha, \beta \in \mathbb{R}\\
					& u(x,0) = u_{0}(x), \quad (x,t)\in\Z^d\times \R,
				\end{aligned}
				\right.
			\end{equation*}
			where $F\in C(\mathbb{R})$. In particular, the DNLS with saturable nonlinearity, in which $F(u)=\frac{1}{1+|u|}$, has a unique and global classical solution. For more studies on the saturable case, we refer to \cite{36,37,38}.
		\end{rem}

		\section{The smooth dependence and analytic dependence}
		In this section, we will show that the solution to the equation (\ref{NLS}) depends smoothly on the initial data, when $p$ is an odd positive integer, which we call the algebraic case. To specifically illustrate this idea, we need to briefly recall the definition of the Fr\'{e}chet derivative and its higher order extension, see pages 619, 627 in \cite{4}.
		\begin{defi}
			Let $U$ and $V$ be normed spaces and map $T: U\to V$ is possibly nonlinear. If there exists a bounded linear operator $T'(u)\in \mathcal{B}(U,V)$, for some $u\in U$, s.t.
			\begin{equation*}
				\lim_{\|\Delta u\|\to 0}\frac{\|T(u+\Delta u)-T(u)-T'(u)\Delta u\|}{\|\Delta u\|}=0,
			\end{equation*}
			for all $\Delta u\in U$, then $T'(u)$ is called the Fr\'{e}chet derivative of $T$ at $u$. The second order Fr\'{e}chet derivative of $T$ is defined as the Fr\'{e}chet derivative of $T'$ and so on and so forth. Equivalently, we can express the higher order Fr\'{e}chet derivatives as multilinear maps, which provide a more explicit idea of the higher order Fr\'{e}chet derivatives. 
			\par 
			Suppose the Fr\'{e}chet derivative exists, then we define the second order Fr\'{e}chet derivative  at $u\in U$  as 
			\begin{equation*}
				T''(u)(v_{1},v_{2})=\lim_{t\to 0}\frac{T'(u+tv_{2})(v_{1})-T'(u)(v_{1})}{t},
			\end{equation*}
			for all $v_{1},v_{2} \in U$, if the above limit exists and $T''(u)\in \mathcal{B}(U \times U, V)$. For convenience, we denote $T^{(1)}=T'$ and $T^{(2)}=T''$, then we inductively define higher order derivatives. Suppose $T^{(n-1)}$ exists, the $n^{th}$-order Fr\'{e}chet derivative at $u\in U$ is defined by
			\begin{equation*}
				T^{(n)}(u)(v_{1},\cdots, v_{n})=\lim_{t\to 0}\frac{T^{n-1}(u+tv_{n})(v_{1},\cdots,v_{n-1})-T^{n-1}(u)(v_{1},\cdots,v_{n-1})}{t},
			\end{equation*}
			for all $v_{1},\cdots, v_{n} \in U$, if the above limit exists and $T^{(n)}\in \mathcal{B}(U^{n}, V)$, where $U^{n}=\underbrace{U\times \cdots \times U}_{n\; times}$.
		\end{defi}
		
		Before we state the definition of $\Psi^{k}$, we shall first introduce some useful notations to simplify their expressions.
		\begin{defi}\label{def1}
			For $n\in \Z^{+}$, $\Upsilon=(\epsilon_{1},\epsilon_{2},\cdots,\epsilon_{n})$ is called $n$-signal, if $\epsilon_{j}\in \lbrace 0,1\rbrace$, $j=1,2,\cdots,n$.
			\par Suppose $1\le k_{1}<k_{2}<\cdots<k_{p-1}<k_{p}=n$, then we call  
			\begin{equation*}
				\tau= \left[(i_{1},\cdots, i_{k_{1}});(i_{k_{1}+1},\cdots, i_{k_{2}});\cdots;(i_{k_{p-1}+1},\cdots, i_{k_{p}})\right],
			\end{equation*}
			the $n$-partition, if $\lbrace i_{1},i_{2},\cdots, i_{k_{p}-1}, i_{k_{p}}\rbrace$=$\lbrace 1,2,\cdots, n\rbrace$. Additionally, we use $\|\tau\|:=p$ to denote the number of groups in this partition.
			\par 
			For example, $\tau=\left[(1,3);(2);(4)\right]$ is 4-partition, with $\|\tau\|=3$ and $\tau'=\left[(1,5,6);(3,2);(4);(7)\right]$ is 7-partition, with $\|\tau'\|=4$. It's noteworthy that we will distinguish the cases with same group but different arrangements, such as 
			\begin{equation*}
				\left[(1,3);(2);(4)\right]\ne \left[(3,1);(2);(4)\right],\quad \left[(1,5,6);(3,2);(4);(7)\right]\ne \left[(1,5,6);(3,2);(7);(4)\right].
			\end{equation*}
			We define the set $\widetilde{\tau}$, which consists of $p$ elements with the following form
			\begin{equation*}
				\left[(i_{1},\cdots, i_{k_{1}});\cdots;(i_{k_{m-1}+1},\cdots,i_{k_{m}},n+1);\cdots;(i_{k_{p-1}+1},\cdots, i_{k_{p}})\right],
			\end{equation*}
			where $1\le m\le p$. For example, if $\tau=\left[(1,3);(2);(4)\right]$, then $\widetilde{\tau}$ is the set 
			\begin{equation*}
				\lbrace \left[(1,3,5);(2);(4)\right], \left[(1,3);(2,5);(4)\right], \left[(1,3);(2);(4,5)\right]\rbrace.
			\end{equation*}
		\end{defi}
		\begin{defi}\label{def2}
			For $s\in[0,1]$, suppose $U=\ell^{2,s}(\Z^{d})$, $V=C([0,T];\ell^{2,s}(\Z^{d}))$, $L=\lbrace L^{k} \rbrace_{k=1}^{\infty}$ is a family of multilinear maps, where $L^{k}: U^{k}\to V$ is $k$-linear map. Then for $n$-partition $\tau$ and $\|\tau\|$-signal $\Upsilon$
			\begin{equation*}
				\tau=\left[(i_{1},\cdots, i_{k_{1}});(i_{k_{1}+1},\cdots, i_{k_{2}});\cdots;(i_{k_{p-1}+1},\cdots, i_{k_{p}})\right],\quad  \Upsilon=(\epsilon_{1},\epsilon_{2},\cdots,\epsilon_{p}),
			\end{equation*}
			we define $L^{\tau, \Upsilon}:U^{n}\to V$, for all $v_{1},v_{2},\cdots, v_{n}\in U$, as follows,
			\begin{equation*}
				L^{\tau, \Upsilon}(v_{1},v_{2},\cdots, v_{n}):=[L^{k_{1}}(v_{i_{1}},\cdots,v_{i_{k_{1}}})]^{(\epsilon_{1})}\cdots[L^{k_{p}-k_{p-1}}(v_{i_{k_{p-1}+1}},\cdots,v_{i_{k_{p}}})]^{(\epsilon_{p})},
			\end{equation*}
			where $[A]^{(\epsilon)}$ means $A$, if $\epsilon=0$, or $\overline{A}$, if $\epsilon=1$. For example, if $\tau=\left[(1,3);(2);(4)\right]$ and $\Upsilon=(0,1,1)$, then we have 
			\begin{equation}
				L^{\tau, \Upsilon}(v_{1},v_{2},v_{3},v_{4})=L^{2}(v_{1},v_{3})\cdot \overline{L^{1}(v_{2})}\cdot \overline{L^{1}(v_{4})}.
			\end{equation}
			
		\end{defi}
		\begin{defi}\label{higher order}
			For $s\in[0,1]$, $u_{0}\in \ell^{2,s}(\Z^{d})$, the $k$-linear map $\Psi^{k}(u_{0}):\underbrace{\ell^{2,s}(\Z^{d})\times\cdots \times \ell^{2,s}(\Z^{d})}_{k\; times} \to C([0,T];\ell^{2,s}(\Z^{d}))$, for convenience we sometimes denote it as $\Psi_{u_{0}}^{k}$, is defined by following induction. For $k=1$, we define $\Psi_{u_{0}}^{1}$ as the solution map of the equation (\ref{r}), i.e. $\forall v_{1}\in \ell^{2,s}(\Z^{d})$, $\Psi_{u_{0}}^{1}(v_{1})$ satisfies
			\begin{equation}\label{r}
				\left\{
				\begin{aligned}
					& i\partial_{t} \Psi_{u_{0}}^{1}(v_{1})(x,t)+ \dfrac{1}{2}\Delta \Psi_{u_{0}}^{1}(v_{1})(x,t) =\mu\left(\frac{p+1}{2}|\Phi(u_{0})|^{p-1}\cdot \Psi_{u_{0}}^{1}(v_{1})+\frac{p-1}{2}|\Phi(u_{0})|^{p-3}\Phi(u_{0})^{2}\cdot \overline{\Psi_{u_{0}}^{1}(v_{1})}\right),\\
					& \Psi_{u_{0}}^{1}(v_{1})(x,0) = v_{1}(x), \quad (x,t)\in\Z^d\times \R.
				\end{aligned}
				\right.
			\end{equation}
			Suppose $\Psi_{u_{0}}^{k}$ is defined and for any $ v_{1},\cdots,v_{k-1},v_{k}\in \ell^{2,s}(\Z^{d})$, $\omega=\Psi_{u_{0}}^{k}(v_{1},\cdots,v_{k-1},v_{k})$ satisfies the following equation (\ref{i})
			\begin{equation}\label{i}
				\left\{
				\begin{aligned}
					& i\partial_{t} \omega(x,t)+ \dfrac{1}{2}\Delta \omega(x,t) =\mu\left(\sum_{\tau^{k}}\sum_{\Upsilon}F_{\tau^{k},\Upsilon}(\Phi(u_{0}))\Psi_{u_{0}}^{\tau^{k},\Upsilon}\right),\\
					& \omega(x,0) = p_{k}(x), \quad (x,t)\in\Z^d\times \R,
				\end{aligned}
				\right.
			\end{equation}
			where we follow the notations in Definition \ref{def1} and Definition \ref{def2}, with $L=\Psi_{u_{0}}, L^{k}=\Psi_{u_{0}}^{k}$. The first summation in the equation (\ref{i}) is over all $k$-partitions, and the second summation is over all $\|\tau^{k}\|$-signals.
			\par
			Then we define $\Psi_{u_{0}}^{k+1}$ as the solution map of the equation (\ref{u}), i.e. for any $v_{1},\cdots,v_{k},v_{k+1}\in \ell^{2,s}(\Z^{d})$, $\omega=\Psi_{u_{0}}^{k+1}(v_{1},\cdots,v_{k},v_{k+1})$ satisfies
			\begin{equation}\label{u}
				\left\{
				\begin{aligned}
					& i\partial_{t} \omega(x,t)+ \dfrac{1}{2}\Delta \omega(x,t) =\mu\left(\sum_{\tau^{k}}\sum_{\Upsilon}\left(\frac{d F_{\tau^{k},\Upsilon}}{dz}\Psi_{u_{0}}^{1}(v_{k+1})+\frac{d F_{\tau^{k},\Upsilon}}{d\overline{z}}\overline{\Psi_{u_{0}}^{1}(v_{k+1})}\right)\Psi_{u_{0}}^{\tau^{k},\Upsilon}+(\ast)_{\tau^{k},\Upsilon}\right),\\
					& \omega(x,0) = 0, \quad (x,t)\in\Z^d\times \R.
				\end{aligned}
				\right.
			\end{equation}
			The last term $(\ast)_{\tau^{k},\Upsilon}$ is given by $(\ast)_{\tau^{k},\Upsilon}=\sum_{\tau'\in\widetilde{\tau^{k}} }F_{\tau^{k},\Upsilon}(\Phi(u_{0}))\Psi_{u_{0}}^{\tau',\Upsilon}$, where the meaning of  notation $\widetilde{\tau^{k}}$ follows from Definition (\ref{def1}).
		\end{defi}
		\begin{rem}
			To better illustrate the above notations in definition, we will calculate $\Psi_{u_{0}}^{2}(v_{1},v_{2})$, for any $v_{1},v_{2}\in\ell^{2,s}(\Z^{d})$, with $p$ big enough to avoid trivialty. First, we notice that, in this case, there are only one partition $\tau^{1}$ and two signals $\Upsilon$ as follows
			\begin{equation*}
				\tau^{1}=[(1)], \quad \Upsilon=(0) \;or\; (1).
			\end{equation*}
			The corresponding $F_{\tau^{1},\Upsilon}(\Phi(u_{0}))$ and $\Psi_{u_{0}}^{\tau^{1},\Upsilon}$ are given by
			\begin{equation*}
				F_{\tau^{1},\Upsilon}(\Phi(u_{0}))=\frac{p+1}{2}|\Phi(u_{0})|^{p-1}\; or \;\frac{p-1}{2}|\Phi(u_{0})|^{p-3}\Phi(u_{0})^{2},\quad \Psi_{u_{0}}^{\tau^{1},\Upsilon}=\Psi_{u_{0}}^{1}(v_{1})\; or \;\overline{\Psi_{u_{0}}^{1}(v_{1})}.
			\end{equation*}
			Thus, from the definition and calculation, we have the following correspondences
			\begin{equation*}
				\frac{d F_{\tau^{k},\Upsilon}}{dz}=\frac{d F_{\tau^{k},\Upsilon}}{dz}(\Phi(u_{0}))=\frac{p^{2}-1}{4}|\Phi(u_{0})|^{p-3}\overline{\Phi(u_{0})}\; or\; \frac{p^{2}-1}{4}|\Phi(u_{0})|^{p-3}\Phi(u_{0}).
			\end{equation*}
			The calculation of $\frac{d F_{\tau^{k},\Upsilon}}{d\overline{z}}$ is similar, we omitted. For the term in $(\ast)_{\tau^{1}, \Upsilon}$, we have $\widetilde{\tau^{1}}=\lbrace [(1,2)]\rbrace$ and the following correspondences 
			\begin{equation*}
				(\ast)_{\tau^{1}, \Upsilon}=\frac{p+1}{2}|\Phi(u_{0})|^{p-1}\Psi_{u_{0}}^{2}(v_{1},v_{2})\left(=\frac{p+1}{2}|\Phi(u_{0})|^{p-1}\omega\right)\; 
			\end{equation*}
			\begin{equation*}
				or
			\end{equation*}
			\begin{equation*}
				=\frac{p-1}{2}|\Phi(u_{0})|^{p-3}\Phi(u_{0})^{2} \overline{\Psi_{u_{0}}^{2}(v_{1},v_{2})}\left(=\frac{p-1}{2}|\Phi(u_{0})|^{p-3}\Phi(u_{0})^{2}\overline{\omega}\right).
			\end{equation*}
			Then substituting the above calculation into the equation (\ref{u}), we derive the solution map given by $\Psi_{u_{0}}^{2}$.
		\end{rem}
		Then we state the relationship between $\Phi$ and $\Psi$, which will lead a smooth dependence on the initial data. 
		
		\begin{proof}[Proof of Theorem \ref{3}]
			We just check the case $k=1,s=0$, as the rest of the cases are similar and even simpler. From the definition of the Fr\'{e}chet derivative, we just need to prove that for $v_{0},u_{0}\in \ell^{2}(\Z^{d})$ and $\|v_{0}-u_{0}\|_{\ell^{2}(\Z^{d})}\to 0$, we have
			\begin{equation}\label{y}
				\|\Phi(v_{0})-\Phi(u_{0})-\Psi_{u_{0}}^{1}(v_{0}-u_{0})\|_{C([0,T];\ell^{2}(\Z^{d}))}=o(\|v_{0}-u_{0}\|_{\ell^{2}(\Z^{d})}).
			\end{equation}
			For convenience, we denote $\omega:=\Phi(v_{0})-\Phi(u_{0})-\Psi_{u_{0}}^{1}(v_{0}-u_{0})$. 
			\par
			From the Duhamel formula and definition of $\Phi$, we have the following identity
			\begin{equation*}
				\left[\Phi(v_{0})-\Phi(u_{0})\right](x,t)=e^{it\Delta/2}(v_{0}-u_{0})-i\mu\int_{0}^{t}e^{i(t-s)\Delta/2}\left(|\Phi(u_{0})|^{p-1}\Phi(u_{0})-|\Phi(v_{0})|^{p-1}\Phi(v_{0})\right) ds .
			\end{equation*}
			Similarly, using the Duhamel formula of $\Psi_{u_{0}}^{1}(v_{0}-u_{0})$ and the following first order Taylor expansion of $|z|^{p-1}z$ 
			\begin{equation*}
				|z+\Delta z|^{p-1}(z+\Delta z)=|z|^{p-1}z+\frac{p+1}{2}|z|^{p-1}\cdot\Delta z+ \frac{p-1}{2}|z|^{p-3}z^{2}\cdot \overline{\Delta z}+o(|\Delta z|),
			\end{equation*}
			we can obtain the identity as follows
			\begin{equation*}
				\omega(x,t)=-i\mu\int_{0}^{t}e^{i(t-s)\Delta/2}\left(\frac{p+1}{2}|\Phi(u_{0})|^{p-1}\cdot \omega(x,s)+\frac{p-1}{2}|\Phi(u_{0})|^{p-3}\Phi(u_{0})^{2}\cdot \overline{\omega(x,s)}+o(\Phi(v_{0})-\Phi(u_{0}))\right) ds.
			\end{equation*}
			Applying $\ell^{2}(\Z^{d})$-norm on both sides and using the Minkowski inequality, we can deduce that
			\begin{equation}\label{t}
				\|\omega(\cdot,t)\|_{\ell^{2}(\Z^{d})}\lesssim_{u_{0}} \int_{0}^{t}\|\omega(\cdot,s)\|_{\ell^{2}(\Z^{d})}ds+ o_{T}(\|v_{0}-u_{0}\|_{\ell^{2}(\Z^{d})}),
			\end{equation}
			where the last term $o_{T}(\|v_{0}-u_{0}\|_{\ell^{2}(\Z^{d})})$ comes from the energy estimate in Theorem \ref{energy estimate1} and the Gronwall inequality. Then we apply the Gronwall inequality again to (\ref{t}), we can immediately deduce (\ref{y}). 
			\par 
			Next, we will obtain the boundedness of $\Psi_{u_{0}}^{1}$. Still using the Duhamel formula of the equation (\ref{r}) and applying $\ell^{2}(\Z^{d})$-norm on both side, with the Minkowski inequality, we can similarly get 
			\begin{equation*}
				\|\Psi_{u_{0}}^{1}(v_{1})(\cdot,t)\|_{\ell^{2}(\Z^{d})}\lesssim_{u_{0}}\|v_{1}\|_{\ell^{2}(\Z^{d})}+\int_{0}^{t} \|\Psi_{u_{0}}^{1}(v_{1})(\cdot,s)\|_{\ell^{2}(\Z^{d})} ds.
			\end{equation*}
			Thus, the Gronwall inequality will directly yield the boundedness of $\Psi_{u_{0}}^{1}$.
		\end{proof}
		Based on the above smooth dependence result, we can actually establish the analytic dependence result in some sense. Before we formally state this result, we shall briefly recall the famous Cauchy-Kovalevskaya Theorem, see Theorem 1 in \cite{33}.
		\begin{lemma}\label{CK}
			Suppose $T, \delta>0$ and $F:(a-\delta,a+\delta)\to \mathbb{R} $ is real analytic near $a$ and $y=y(x)$ is the unique solution to the following ordinary differential equation (ODE)
			\begin{equation}
				\left\{
				\begin{aligned}
					& \frac{dy}{dx}=F(y(x)),  \\
					& y(0)= a, \quad x\in [-T,T],
				\end{aligned}
				\right.
			\end{equation}
			then $y=y(x)$ is also real analytic near zero.
		\end{lemma}
		
		\begin{proof}[Proof of Theorem \ref{4}]
			For simplicity, we only show the analyticity of $u^{(\varepsilon)}$ in the case $p=3, s=0,\mu=1$, and the general case is parallel. From the definition of the $k^{th}$-order Fr\'{e}chet derivative, we can immediately obtain the following simple identity
			\begin{equation*}
				\left.\frac{d^{k}}{d\varepsilon^{k}} u^{(\varepsilon)}(x,t)\right|_{\varepsilon=\lambda}=\Psi^{k}(\lambda u_{0})(u_{0},\cdots, u_{0}), \; k=0,1,\cdots
			\end{equation*}
			Therefore, Theorem \ref{3} directly implies that $u^{(\varepsilon)}$ is a smooth function with respect to $\varepsilon$. Next, we will use Taylor expansion to express every order derivative of $u^{(\varepsilon)}$. Suppose that we have the Taylor expansion as follows, for convenience we require $0<\varepsilon<1$,
			\begin{equation}\label{o}
				u^{(\lambda+\varepsilon)}(x,t)=u^{(\lambda)}(x,t)+\varepsilon u_{1}^{(\lambda)}(x,t)+\varepsilon^{2} u_{2}^{(\lambda)}(x,t)+\cdots+\varepsilon^{n} u_{n}^{(\lambda)}(x,t)+o(\varepsilon^{n}).
			\end{equation}
			Inserting (\ref{o}) into equation (\ref{p}) and comparing coefficients, we can obtain the following equations (for simplicity, we just state first three equations, which are representative enough)
			\begin{equation}\label{a}
				\left\{
				\begin{aligned}
					& i\partial_{t} u^{(\lambda)}(x,t)+ \dfrac{1}{2}\Delta u^{(\lambda)}(x,t) =(u^{(\lambda)})^{2}\overline{u^{(\lambda)}},  \\
					& u^{(\lambda)}(x,0) = \lambda u_{0}(x), \quad (x,t)\in\Z^d\times \R,
				\end{aligned}
				\right.
			\end{equation} 
			\begin{equation}\label{s}
				\left\{
				\begin{aligned}
					& i\partial_{t} u_{1}^{(\lambda)}(x,t)+ \dfrac{1}{2}\Delta u_{1}^{(\lambda)}(x,t)=(u^{(\lambda)})^{2}\overline{u_{1}^{(\lambda)}}+2u^{(\lambda)}\overline{u^{(\lambda)}}u_{1}^{(\lambda)},  \\
					& u_{1}^{(\lambda)}(x,0) = u_{0}(x), \quad (x,t)\in\Z^d\times \R,
				\end{aligned}
				\right.
			\end{equation} 
			\begin{equation}\label{d}
				\left\{
				\begin{aligned}
					& i\partial_{t} u_{2}^{(\lambda)}(x,t)+ \dfrac{1}{2}\Delta u_{2}^{(\lambda)}(x,t)=(u^{(\lambda)})^{2}\overline{u_{2}^{(\lambda)}}+2u^{(\lambda)}\overline{u^{(\lambda)}}u_{2}^{(\lambda)}+(u_{1}^{(\lambda)})^{2}\overline{u^{(\lambda)}}+2u_{1}^{(\lambda)}\overline{u_{1}^{(\lambda)}}u^{(\lambda)},  \\
					& u_{2}^{(\lambda)}(x,0) = 0, \quad (x,t)\in\Z^d\times \R.
				\end{aligned}
				\right.
			\end{equation}
			Next, we inductively control the $\ell^{2}(\Z^{d})$-norm of each $u_{k}$. Choosing $R>0$, s.t. $\lambda, \|u_{0}\|_{\ell^{2}(Z^{d})}\ll R$. First, applying the mass conservation in Theorem 1.1 into the above equation (\ref{a}), we notice that $\|u^{(\lambda)}\|_{C([0,T];\ell^{2}(\Z^{d}))}=\lambda\|u_{0}\|_{\ell^{2}(\Z^{d})}\le C_{0}$, where $C_{0}$ is some constant only depends on $R$. Using the Duhamel formula of the equation (\ref{s}), we can obtain
			\begin{equation}\label{f}
				u_{1}^{(\lambda)}(x,t)=e^{it\Delta/2}u_{0}-i\mu\int_{0}^{t}e^{i(t-s)\Delta/2}\left[(u^{(\lambda)})^{2}\overline{u_{1}^{(\lambda)}}+2u^{(\lambda)}\overline{u^{(\lambda)}}u_{1}^{(\lambda)}\right] ds.
			\end{equation}
			Then, taking $\ell^{2}(\Z^{d})$-norm on both sides of (\ref{f}) and using the Minkowski inequality, we obtain the inequality as follows,
			\begin{equation}\label{g}
				\|u_{1}^{(\lambda)}(\cdot,t)\|_{\ell^{2}(\Z^{d})}\lesssim_{R}\|u_{0}\|_{\ell^{2}(\Z^{d})}+\int_{0}^{t}\|u_{1}^{(\lambda)}(\cdot,s)\|_{\ell^{2}(\Z^{d})}ds.
			\end{equation}
			Applying the Gronwall inequality on (\ref{g}), we can immediately conclude that there exists some constant $C_{1}$, s.t. $\|u_{1}^{(\lambda)}\|_{C([0,T];\ell^{2}(\Z^{d}))}\le C_{1}\|u_{0}\|_{\ell^{2}(\Z^{d})}$, where $C_{1}$ only depends on $R,T$. 
			\par 
			Next, we suppose that $\|u_{k}^{(\lambda)}\|_{C([0,T];\ell^{2}(\Z^{d}))}\le C_{k}\|u_{0}\|_{\ell^{2}(\Z^{d})}^{k}, \; k\le n$, where we regard $u_{0}^{\lambda}$ as $u^{(\lambda)}$. Then we need to estimate $\|u_{n+1}^{(\lambda)}\|_{C([0,T];\ell^{2}(\Z^{d}))}$ by $C_{n+1}\|u_{0}\|_{\ell^{2}(\Z^{d})}^{n+1}$, with some suitable constant $C_{n+1}$. 
			\par 
			From the similar estimates as in (\ref{f})-(\ref{g}) and simple product estimate as follows
			\begin{equation*}
				\|\prod_{i=1}^{m}f_{j}\|_{\ell^{2}(\Z^{d})}\le \prod_{i=1}^{m}\|f_{j}\|_{\ell^{2}(\Z^{d})},
			\end{equation*}
			we can derive the estimate $\|u_{n+1}^{(\lambda)}\|_{C([0,T];\ell^{2}(\Z^{d}))}\le C_{n+1}\|u_{0}\|_{\ell^{2}(\Z^{d})}^{n+1}$ by the appearance law of the inhomogeneous terms in the equations (\ref{a})-(\ref{d}) and induction, where constant $C_{n+1}$ is defined as follows
			\begin{equation}\label{h}
				C_{n+1}=K\sum_{k_{1}+k_{2}+k_{3}=n+1,\; 0\le k_{i}<n+1}C_{k_{1}}C_{k_{2}}C_{k_{3}}, \quad n\ge 1,
			\end{equation}
			for some constant $K>0$ only depends on $R,T$. Then we just need to show that the following series, given by the identity (\ref{h}),
			\begin{equation}\label{j}
				f(x)\equiv \sum_{k=0}^{\infty}C_{k}x^{k},
			\end{equation}
			will converge in $x\in (-\delta,\delta)$, for some $\delta>0$. Indeed, to show the Taylor expansion of $u^{(\lambda+\varepsilon)}$ is equal to  $u^{(\lambda+\varepsilon)}$ itself, it suffices to prove the Lagrange remainder 
			\begin{equation*}
				\frac{\varepsilon^{n+1}}{(n+1)!}\left.\frac{d^{n+1}}{d\varepsilon^{n+1}}u^{(\varepsilon)}(x,t)\right|_{\varepsilon=\lambda'}=\varepsilon^{n+1}u_{n+1}^{(\lambda')}(x,t), \quad \lambda<\lambda'<\lambda+\varepsilon,
			\end{equation*} 
			will converge to zero as $n\to \infty$. Notice that the above estimate is uniformly hold for any $u^{(\lambda')}$, as the choice of $\lbrace C_{k} \rbrace$ only depends on $R,T$ and is consistent for $ \lambda<\lambda'<\lambda+\varepsilon$.
			\par 
			Therefore, if we can show the convergence of (\ref{j}), then by the d'Alembert criterion, we can immediately derive
			\begin{equation*}
				|u_{n+1}^{(\lambda')}(x,t)|^{\frac{1}{n+1}}\le\|u_{n+1}^{(\lambda')}\|_{C([0,T];\ell^{2}(\Z^{d}))} ^{\frac{1}{n+1}}\le (C_{n+1})^{\frac{1}{n+1}}\cdot \|u_{0}\|_{\ell^{2}(\Z^{d})}\le \frac{1}{\delta}\|u_{0}\|_{\ell^{2}(\Z^{d})}.
			\end{equation*}
			Thus, for any $\varepsilon<\frac{\delta}{\|u_{0}\|_{\ell^{2}(\Z^{d})}}$, the Lagrange remainder will converge to zero as $n\to \infty$, which ensures the analyticity of $u^{\varepsilon}$ at point $\lambda$.
			\par 
			Notice that $y=f(x)$ is the implicit function of the following equation
			\begin{equation}\label{kk}
				Ay-By^{3}=Cx+D,
			\end{equation} 
			where the constant $A,B,C,D$ are given by
			\begin{equation*}
				A:=3KC_{0}^{2}+1,\quad B:=K, \quad C:=C_{1},\quad D:=2KC_{0}^{3}+C_{0}.
			\end{equation*}
			Differentiating the equation (\ref{kk}) with respect to $x$, we can immediately show that $y$ satisfies the following ODE
			\begin{equation}\label{l}
				\left\{
				\begin{aligned}
					& \frac{dy}{dx}=\frac{C}{A-3By^{2}},  \\
					& y(0)= \gamma, \quad x\in \R,
				\end{aligned}
				\right.
			\end{equation}
			where $\gamma$ satisfies $A\gamma-B\gamma^{3}=D$. To ensure that $A-3By^{2}>0$ in some neighborhood of $x=0$, it suffices to check $D\le \frac{2A}{3}\sqrt{\frac{A}{3B}}$, which is a direct result of calculation. Thus, we can apply Lemma \ref{CK} to ODE (\ref{l}) and derive the analyticity of solution $y$ in some neighborhood of $x=0$. 
			\par We now finish the whole proof of the analyticity of $u^{\varepsilon}$ for all fixed $u_{0}\in \ell^{2}(\Z^{d})$.
		\end{proof}
		\begin{rem}
			From the proof of Theorem \ref{4}, we can further conclude that the lower bound of convergence radius of $u^{\varepsilon}(x,t)$ is uniform in $x\in \Z^{d}$ and $t$ in any compact time interval. 
		\end{rem}
		\section{The Scattering theory of discrete semilinear Sch\"{o}dinger equations}
		In this section, we will use the Strichartz estimates to establish the existence of the wave operator and the asymptotic completeness in the energy space $\ell^{2}(\Z^{d})$.
		\par We first state the following useful Strichartz estimates, see Theorem 1.2 in \cite{6}, which is induced by the decay estimate of propagator $e^{it\Delta/2}$ in Theorem 1 from \cite{5}.
		\begin{thm}
			For discrete free Sch\"{o}dinger equation 
			\begin{equation}\label{free}
				\left\{
				\begin{aligned}
					& i\partial_{t}u+\frac{1}{2}\Delta u=0,  \\
					& u(x,0)= u_{0}(x), \quad (x,t)\in\Z^d\times \R,
				\end{aligned}
				\right.
			\end{equation}
			there exist two estimates as follows
			\begin{equation*}
				\|u(\cdot,t)\|_{\ell^{\infty}(\Z^{d})}\lesssim \frac{1}{|t|^{d/3}}\|u_{0}\|_{\ell^{1}(\Z^{d})}, \quad \|u(\cdot,t)\|_{\ell^{2}(\Z^{d})}=\|u_{0}\|_{\ell^{2}(\Z^{d})}.
			\end{equation*}
		\end{thm}
		\begin{coro}\label{Strichartz}
			We call the exponent pair $(q,r)$ is admissible, if the pair satisfies
			\begin{equation*}
				q,r\ge 2, \quad (q,r,d)\ne(2,\infty,3),\quad \frac{1}{q}+\frac{d}{3r}\le \frac{d}{6}.
			\end{equation*}
			If the equality holds in the last inequality above, we say the exponent pair $(q,r)$ is sharp admissible.
			Then there exist the following three Strichartz estimates 
			\begin{equation*}
				\|e^{it\Delta/2}f\|_{L_{t}^{q}\ell_{x}^{r}(\mathbb{R}\times \Z^{d})}\lesssim \|f\|_{\ell^{2}(\Z^{d})},
			\end{equation*}
			\begin{equation*}
				\|\int_{\mathbb{R}}e^{-s\Delta/2}F(\cdot,s)ds\|_{\ell^{2}(\Z^{d})}\lesssim \|F\|_{L_{t}^{q'}\ell_{x}^{r'}(\mathbb{R}\times \Z^{d})},
			\end{equation*}
			\begin{equation*}
				\|\int_{s<t}e^{i(t-s)\Delta/2}F(\cdot,s)ds\|_{L_{t}^{q}\ell_{x}^{r}(\mathbb{R}\times \Z^{d})}\lesssim \|F\|_{L_{t}^{\widetilde{q}'}\ell_{x}^{\widetilde{r}'}(\mathbb{R}\times \Z^{d})},
			\end{equation*}
			where $(q,r), (\widetilde{q},\widetilde{r})$ are admissible exponent pairs. 
		\end{coro}
		For convenience, we introduce the Strichartz space and its dual space, which can state the above Strichartz estimates in a uniform manner.
		\begin{defi}\label{m}
			For some time interval $I$, the Strichartz space $S(I\times \Z^{d})$ is defined as the closure of the Schwartz function with the following norm
			\begin{equation*}
				\|v\|_{S(I\times \Z^{d})}:=\sup_{(q,r) \; admissible}\|v\|_{L_{t}^{q}\ell_{x}^{r}(I\times \Z^{d})}.
			\end{equation*}
			And we define $N(I\times \Z^{d})$ as the dual space of Strichartz space $S(I\times \Z^{d})$. By definition, we see that
			\begin{equation*}
				\|F\|_{N(I\times \Z^{d})}\le \|F\|_{L_{t}^{q'}\ell_{x}^{r'}(I\times \Z^{d})}.
			\end{equation*}
			Thus, the Strichartz estimates in Corollary \ref{Strichartz} can be rewritten as the following uniform estimates, with some positive constant $C_{0}$,
			\begin{equation*}
				\|e^{it\Delta/2}f\|_{S(\mathbb{R}\times \Z^{d})}\le C_{0} \|f\|_{\ell^{2}(\Z^{d})},
			\end{equation*}
			\begin{equation*}
				\|\int_{\mathbb{R}}e^{-s\Delta/2}F(\cdot,s)ds\|_{\ell^{2}(\Z^{d})}\le C_{0} \|F\|_{N(\mathbb{R}\times \Z^{d})},
			\end{equation*}
			\begin{equation*}
				\|\int_{s<t}e^{i(t-s)\Delta/2}F(\cdot,s)ds\|_{S(\mathbb{R}\times \Z^{d})}\le C_{0} \|F\|_{N(\mathbb{R}\times \Z^{d})}.
			\end{equation*}
		\end{defi}
		Next, we use the above Strichartz estimates to give a proof of the existence and the continuity of the wave operator.
		\begin{proof}[Proof of Theorem \ref{5}]
			From the Duhamel formula and the definition of the scattering, it suffices to solve the following equation
			\begin{equation}\label{b}
				u(x,t)=e^{it\Delta/2}u_{+}+i\mu \int_{t}^{\infty}e^{i(t-s)\Delta/2}\left(|u(\cdot,s)|^{p-1}u(\cdot,s)\right) ds, \quad t\in \left[0,+\infty\right).
			\end{equation}
			To solve the above equation (\ref{b}), we decompose the whole question into two parts. Fixing $T$ large enough, we first solve the equation (\ref{b}) in the time interval $\left[T,+\infty\right)$ and then extend it to the whole time interval $\left[0,+\infty\right)$.
			\par With the assumption $6\le (p-1)d$, it's not hard to verify the existence of exponents $0<q_{i}, r_{i}<\infty, i=1,2,3,4,5$ satisfying the following conditions
			\begin{equation*}
				\frac{1}{q_{4}}+\frac{1}{q_{5}}=\frac{1}{q_{3}'}, \quad \frac{1}{r_{4}}+\frac{1}{r_{5}}=\frac{1}{r_{3}'},
			\end{equation*} 
			\begin{equation*}
				q_{1}=(p-1)q_{4}, \quad r_{1}=(p-1)r_{4},
			\end{equation*}
			and $(q_{1},r_{1}), (q_{2},r_{2}), (q_{3},r_{3}), (q_{5},r_{5}), (pq_{2}',pr_{2}')$ are admissible exponent pairs. Indeed, we can take 
			\begin{equation*}
				q_{1}=r_{1}=q_{3}=r_{3}=q_{5}=r_{5}=p+1,\quad q_{2}=r_{2}=\frac{2(d+3)}{d},\quad q_{4}=r_{4}=\frac{p+1}{p-1}.
			\end{equation*} 
			Then we introduce the following norm, which is obviously controlled by the Strichartz norm 
			\begin{equation*}
				\|v\|_{S_{0}(I\times \Z^{d})}:=\|v\|_{L_{t}^{q_{1}}\ell_{x}^{r_{1}}(I\times \Z^{d})}+\|v\|_{L_{t}^{pq_{2}'}\ell_{x}^{pr_{2}'}(I\times \Z^{d})},
			\end{equation*}
			where $I$ is a time interval. Then, from the Strichartz estimates in Corollary \ref{Strichartz}, we can fix $T$ large enough that 
			\begin{equation*}
				\|e^{it\Delta/2}u_{+}\|_{S_{0}(\left[T,+\infty\right)\times \Z^{d})}\le \varepsilon,
			\end{equation*}
			where $\varepsilon$ will be determined later. Then we will use the iteration below to construct a solution to the equation (\ref{b})
			\begin{equation}\label{n}
				u_{n}(x,t)=e^{it\Delta/2}u_{+}+i\mu \int_{t}^{\infty}e^{i(t-s)\Delta/2}\left(|u_{n-1}(\cdot,s)|^{p-1}u_{n-1}(\cdot,s)\right) ds, \quad n\ge 0,
			\end{equation}
			where we set $u_{-1}\equiv0$. For convenience, we introduce the following maps $D,N$ to simplify our next discussion
			\begin{equation*}
				D: F(x,t)\mapsto i\mu \int_{t}^{\infty}e^{i(t-s)\Delta/2}F(\cdot,s) ds,
			\end{equation*} 
			\begin{equation*}
				N: v(x,t)\mapsto |v(x,t)|^{p-1}v(x,t).
			\end{equation*}
			Then the equation (\ref{n}) can be equivalently written as 
			\begin{equation}\label{m}
				u_{n}(x,t)=e^{it\Delta/2}u_{+}+DN(u_{n-1}).
			\end{equation}
			Next, we claim that if $\varepsilon$ small enough, then the following propositions are true
			\begin{equation*}
				(P_{n}): \quad \|u_{n}\|_{S_{0}(\left[T,+\infty\right)\times\Z^{d})}\le 4\varepsilon,
			\end{equation*}
			\begin{equation*}
				\widetilde{(P_{n})}: \quad \|N(u_{n})\|_{L_{t}^{q_{2}'}\ell_{x}^{r_{2}'}(\left[T,+\infty\right)\times \Z^{d})}=\||u_{n}|^{p-1}u_{n}\|_{L_{t}^{q_{2}'}\ell_{x}^{r_{2}'}(\left[T,+\infty\right)\times \Z^{d})}\le \frac{\varepsilon}{C_{0}},
			\end{equation*}
			where $C_{0}$ is the constant in Strichartz estimates, see Definition \ref{m}.
			\par 
			First, from the choice of $T$, we see that $(P_{0})$ is true, and then we just need to prove the above propositions by induction, with the following chain
			\begin{equation*}
				(P_{n})\Rightarrow \widetilde{(P_{n})}\Rightarrow (P_{n+1}).
			\end{equation*}
			For $(P_{n})\Rightarrow \widetilde{(P_{n})}$, we just need to let $\varepsilon\ll \frac{1}{C_{0}}$ and observe that
			\begin{equation*}
				\|N(u_{n})\|_{L_{t}^{q_{2}'}\ell_{x}^{r_{2}'}(\left[T,+\infty\right)\times \Z^{d})}=\|u_{n}\|_{L_{t}^{pq_{2}'}\ell_{x}^{pr_{2}'}(\left[T,+\infty\right)\times\Z^{d})}^{p}\le (4\varepsilon)^{p}\le \frac{\varepsilon}{C_{0}}.
			\end{equation*}
			For $\widetilde{(P_{n})}\Rightarrow (P_{n+1})$, we apply $S_{0}(\left[T,+\infty\right)\times\Z^{d})$-norm on iteration (\ref{m}) and obtain
			\begin{equation}
				\|u_{n}\|_{S_{0}(\left[T,+\infty\right)\times\Z^{d})}\le \|e^{it\Delta/2}u_{+}\|_{S_{0}(\left[T,+\infty\right)\times\Z^{d})}+\|DN(u_{n})\|_{S_{0}(\left[T,+\infty\right)\times\Z^{d})}
			\end{equation}
			\begin{equation*}
				\le \varepsilon+2C_{0}\|N(u_{n})\|_{L_{t}^{q_{2}'}\ell_{x}^{r_{2}'}(\left[T,+\infty\right)\times \Z^{d})}\le \varepsilon+2C_{0}\cdot \frac{\varepsilon}{C_{0}}< 4\varepsilon.
			\end{equation*}
			Thus, we complete the induction and, from the propositions $(P_{n})$ and $\widetilde{(P_{n})}$, control the norms of $u_{n}$ and $|u_{n}|^{p-1}u_{n}$ on a very small scale, which will be essential in our following contraction argument.
			\par 
			From the H\"{o}lder inequality and the proposition $(P_{n})$, we also notice that
			\begin{equation*}
				\|N(u_{n})-N(u_{n-1})\|_{N(\left[T,+\infty\right)\times \Z^{d}))}\le \|N(u_{n})-N(u_{n-1})\|_{L_{t}^{q_{3}'}\ell_{x}^{r_{3}'}(\left[T,+\infty\right)\times \Z^{d})}
			\end{equation*}
			\begin{equation*}
				\lesssim \||u_{n}|^{p-1}|u_{n}-u_{n-1}|\|_{L_{t}^{q_{3}'}\ell_{x}^{r_{3}'}(\left[T,+\infty\right)\times \Z^{d})}+\||u_{n-1}|^{p-1}|u_{n}-u_{n-1}|\|_{L_{t}^{q_{3}'}\ell_{x}^{r_{3}'}(\left[T,+\infty\right)\times \Z^{d})}
			\end{equation*}
			\begin{equation*}
				\lesssim \left(\|u_{n}\|_{L_{t}^{q_{1}}\ell_{x}^{r_{1}}(\left[T,+\infty\right)\times \Z^{d})}^{p-1}+\|u_{n-1}\|_{L_{t}^{q_{1}}\ell_{x}^{r_{1}}(\left[T,+\infty\right)\times \Z^{d})}^{p-1}\right)\cdot \|u_{n}-u_{n-1}\|_{L_{t}^{q_{5}}\ell_{x}^{r_{5}}(\left[T,+\infty\right)\times\Z^{d})}
			\end{equation*}
			\begin{equation*}
				\lesssim \varepsilon^{p-1}\cdot \|u_{n}-u_{n-1}\|_{S(\left[T,+\infty\right)\times \Z^{d})}.
			\end{equation*}
			Then if $\varepsilon \ll 1$, we can have
			\begin{equation*}
				\|N(u_{n})-N(u_{n-1})\|_{N(\left[T,+\infty\right)\times \Z^{d})}\le \frac{1}{2C_{0}}\cdot\|u_{n}-u_{n-1}\|_{S(\left[T,+\infty\right)\times \Z^{d})}.
			\end{equation*}
			Based on the above estimate, we can immediately have the following contraction
			\begin{equation*}
				\|u_{n+1}-u_{n}\|_{S(\left[T,+\infty\right)\times \Z^{d})}\le \|DN(u_{n})-DN(u_{n-1})\|_{S(\left[T,+\infty\right)\times \Z^{d})}
			\end{equation*}
			\begin{equation*}
				\le C_{0}\|N(u_{n})-N(u_{n-1})\|_{N(\left[T,+\infty\right)\times \Z^{d})}\le \frac{1}{2}\cdot\|u_{n}-u_{n-1}\|_{S(\left[T,+\infty\right)\times \Z^{d})}.
			\end{equation*}
			Thus, $\lbrace u_{n} \rbrace$ is a Cauchy sequence in $S(\left[T,+\infty\right)\times \Z^{d})\subseteq C(\left[T,+\infty\right);\ell^{2}(\Z^{d}))$. We denote its limit as $u$, which solves the equation (\ref{b}). Then we shall show that $u$ solves the original equation (\ref{NLS}). From the expression of identity (\ref{b}), we see that
			\begin{equation*}
				e^{-it\Delta/2}u(\cdot,t)=u_{+}+i\mu \int_{t}^{\infty}e^{-is\Delta/2}\left(|u(\cdot,s)|^{p-1}u(\cdot,s)\right) ds,
			\end{equation*}
			\begin{equation*}
				u_{+}(x)=e^{-iT\Delta/2}u(\cdot,T)-i\mu \int_{T}^{\infty}e^{-is\Delta/2}\left(|u(\cdot,s)|^{p-1}u(\cdot,s)\right) ds.
			\end{equation*}
			Then we combine the above two identities and obtain 
			\begin{equation}\label{qq}
				u(x,t)=e^{i(t-T)\Delta/2}u(\cdot,T)-i\mu \int_{T}^{t}e^{i(t-s)\Delta/2}\left(|u(\cdot,s)|^{p-1}u(\cdot,s)\right)ds,
			\end{equation} 
			which directly shows $u$ satisfies the equation (\ref{NLS}).
			\par 
			Finally, we just need to use the global existence theory of backward discrete semilinear Schr\"{o}dinger equation in Remark \ref{backward} and extend the solution $u$ to the whole time interval $\left[0,+\infty\right)$. Then we can take $u_{0}(x)=u(0)$ and the above proof shows the solution $u$, with initial data $u_{0}$, scatters to $e^{it\Delta/2}u_{+}$.
			\par The uniqueness of such initial data $u_{0}$ comes immediately from the conservation of total mass. Besides, the continuity of the wave operator $\Omega_{+}$ can be similarly deduced from the above iteration.
		\end{proof}
		\begin{rem}
			Notice that, in the discrete setting $\Z^{d}$, if $u\in C(\left[T,+\infty\right);\ell^{2}(\Z^{d}))$ and satisfies the identity (\ref{qq}), then it's actually a classical solution, as $|u(x,t)|^{p-1}u(x,t)$ in left-hand side also belongs to $C(\left[T,+\infty\right);\ell^{2}(\Z^{d}))$.
		\end{rem}
		Before we give the proof of the asymptotic completeness, we shall introduce several key lemmas.
		
		\begin{lemma}\label{today}
			If the exponent pair $(p_{0},q_{0})$ satisfies one of the following conditions
			\begin{equation}\label{condition}
				\left\{
				\begin{aligned}
					& \frac{6+d}{6p}\le \frac{1}{p_{0}}+\frac{d}{3q_{0}},\quad p\le p_{0},q_{0}\le2p,\quad (p_{0},q_{0},d)\ne (2p,p,3), \\
					& p-1<q_{0},p_{0}<\infty,\quad \frac{1}{p-1}\le \frac{1}{p_{0}}+\frac{d}{3q_{0}},\quad \frac{1}{2}\le \frac{d}{6}+\frac{p-1}{p_{0}},\quad 6\le (p-1)d,
				\end{aligned}
				\right.
			\end{equation}
			and $u$, the classical solution to the equation (\ref{NLS}), with the initial data $u_{0}\in \ell^{2}(\Z^{d})$, has boundedness below 
			\begin{equation}\label{cc}
				\|u\|_{L_{t}^{p_{0}}\ell_{x}^{q_{0}}(\left[0,+\infty\right)\times \Z^{d})}<\infty,
			\end{equation}
			then we have the asymptotic completeness of $u$.
		\end{lemma}
		\begin{proof}
			If the exponent pair $(p_{0},q_{0})$ satisfies the first condition, then it's not difficult to check the existence of some admissible exponent pair $(q,r)$ s.t. $(p_{0},q_{0})=(pq',pr')$. 
			\par 
			From the Duhamel formula, the solution $u$ scatters to some asymptotic state $u_{+}$ is equivalent to the conditional convergence of the integral $\int_{0}^{\infty}e^{it\Delta/2}(|u(\cdot,t)|^{p-1}u(\cdot,t))dt$ in $\ell^{2}(\Z^{d})$-norm. Thus, from the Strichartz estimates, it suffices to show $|u|^{p-1}u\in N(\left[0,+\infty\right)\times \Z^{d})$. Note that we have the following estimate
			\begin{equation*}
				\||u|^{p-1}u\|_{N(\left[0,+\infty\right)\times \Z^{d})}\le \||u|^{p-1}u\|_{L_{t}^{q'}\ell_{x}^{r'}(\left[0,+\infty\right)\times \Z^{d})}=\|u\|_{L_{t}^{p_{0}}\ell_{x}^{q_{0}}(\left[0,+\infty\right)\times \Z^{d})}^{p}<\infty.
			\end{equation*}
			Thus, we complete the proof of the first part.
			\par 
			If the exponent pair $(p_{0},q_{0})$ satisfies the second condition, then there exist $p_{i},q_{i}>0, i=1,2,3$ satisfying the following requirements
			\begin{equation}\label{dd}
				\frac{1}{p_{2}}+\frac{1}{p_{3}}=\frac{1}{q_{1}'}, \quad \frac{1}{q_{2}}+\frac{1}{q_{3}}=\frac{1}{r_{1}'}, 
			\end{equation}
			\begin{equation}\label{ff}
				p_{0}=(p-1)p_{2}, \quad q_{0}=(p-1)q_{2},
			\end{equation}
			where the exponent pairs $(q_{1},r_{1}), (p_{3},q_{3})$ are admissible. Indeed, the second condition is equivalent to 
			\begin{equation}\label{aa}
				1<q_{2},p_{2}<\infty,\quad 1\le \frac{1}{p_{2}}+\frac{d}{3q_{2}},\quad \frac{1}{2}\le \frac{d}{6}+\frac{1}{p_{2}},\quad 6\le (p-1)d,
			\end{equation}
			and, to prove the requirements (\ref{dd}) and (\ref{ff}), it's sufficient to show the existence of $p_{1},q_{1}>0 $ satisfying the following requirements
			\begin{equation*}
				\max\lbrace\frac{1}{2}-\frac{1}{p_{2}},0\rbrace\le\frac{1}{q_{1}}\le\min\lbrace \frac{1}{2}, 1-\frac{1}{p_{2}}\rbrace, \quad \max \lbrace\frac{1}{2}-\frac{1}{q_{2}},0\rbrace\le \frac{1}{r_{1}}\le  \min\lbrace \frac{1}{2}, 1-\frac{1}{q_{2}}\rbrace,
			\end{equation*}
			\begin{equation*}
				1+\frac{d}{6}-(\frac{1}{p_{2}}+\frac{d}{3q_{2}})\le \frac{1}{q_{1}}+\frac{d}{3r_{1}}\le \frac{d}{6}.
			\end{equation*}
			Based on condition (\ref{aa}), we can use the intermediate value theorem to show the existence of such $p_{1},q_{1}$.
			\par 
			From the assumption $6\le (p-1)d$ in the second condition, we can obtain some admissible exponent pair $(q,r)$, such that $(pq',pr')$ is also admissible. Then, just like the proof in the first part, we can derive that
			\begin{equation*}
				\||u|^{p-1}u\|_{N(\left[0,+\infty\right)\times \Z^{d})}\le \|u\|_{S(\left[0,+\infty\right)\times \Z^{d})}^{p}.
			\end{equation*}
			Therefore, it suffices to show the boundedness of $\|u\|_{S(\left[0,+\infty\right)\times \Z^{d})}$.  From the assumption $p_{0}<\infty$, we can decompose the whole positive real axis $\left[0,+\infty\right)$ into a finite number of sub-intervals, s.t. for each sub-interval $I$, we have 
			\begin{equation*}
				\|u\|_{L_{t}^{p_{0}}\ell_{x}^{q_{0}}(I\times \Z^{d})}\le \varepsilon,
			\end{equation*}
			where $\varepsilon>0$ to be determined later. Fixing one of the above sub-intervals $I$ and 
			applying the Strichartz estimates and the H\"{o}lder inequality, we can obtain the following estimate
			\begin{equation*}
				\|u\|_{S(I\times \Z^{d})}\lesssim \|u_{0}\|_{\ell^{2}(\Z^{d})}+\||u|^{p-1}u\|_{L_{t}^{q_{1}'}\ell_{x}^{r_{1}'}(I\times \Z^{d})}
			\end{equation*}
			\begin{equation*}
				\lesssim \|u_{0}\|_{\ell^{2}(\Z^{d})}+ \|u\|_{L_{t}^{p_{0}}\ell_{x}^{q_{0}}(I\times \Z^{d})}^{p-1}\|u\|_{L_{t}^{p_{3}}\ell_{x}^{q_{3}}(I\times \Z^{d})}\lesssim \|u_{0}\|_{\ell^{2}(\Z^{d})}+\varepsilon^{p-1}\|u\|_{S(I\times \Z^{d})}. 
			\end{equation*}
			Then we can take $\varepsilon\ll 1$ and deduce that
			\begin{equation*}
				\|u\|_{S(I\times \Z^{d})\lesssim_{u_{0}}} 1.
			\end{equation*}
			Summing over all sub-intervals, we can obtain the boundedness of $\|u\|_{S(\left[0,+\infty\right)\times \Z^{d})}$.
		\end{proof}
		
		\begin{rem}
			Obviously, the exponent pair $(p_{0},q_{0})$ satisfies one of two conditions in (\ref{condition}) exists. For example, when $d=p=3$, exponent pairs $(p_{0},q_{0})=(4,4),(\frac{5}{2},5)$ can satisfy the second condition. Another immediate but important fact is that the second condition is stronger than the first condition, when $3\le d$ and $6\le (p-1)d$. 
		\end{rem}
	   Next, we introduce a key lemma from Theorem 3 in \cite{5}, which shows the stability of a global classical solution with the small initial data.
	   \begin{lemma}\label{1003}
	   	If $6\le (p-1)d$, there exist $\varepsilon=\varepsilon(d)>0$ and constant $C=C(d)$, such that whenever $\|u_{0}\|_{\ell^{2}(\Z^{d})}\le \varepsilon$, the corresponding global solution $u$ satisfies the following bounds
	   	\begin{equation*}
	   		\|u\|_{L_{t}^{q}\ell_{x}^{r}(\left[0,+\infty\right)\times \Z^{d})}\le C\varepsilon,
	   	\end{equation*}
	   	for all admissible pairs $(q,r)$.
	   \end{lemma} 
		\begin{proof}[Proof of Theorem \ref{6}]
			From Lemma \ref{today} and Lemma \ref{1003}, it suffices to find admissible pairs $(p_{0},q_{0})$ satisfying the second condition in Lemma \ref{today}. 
			\par For the case $d\le 3$, we can take the ternary pairs $(d,p_{0},q_{0})$ as follows
			\begin{equation*}
				(d,p_{0},q_{0})=(1,\frac{6(p-1)}{5},2(p-1)),\quad (2,2(p-1),\frac{4(p-1)}{3}),\quad (3,2(p-1),2(p-1)).
			\end{equation*}
			\par For the remaining case $d\ge 4$, we let the ternary pairs $(d,p_{0},q_{0})$ be 
			\begin{equation*}
				(d,p_{0},q_{0})=\left\{
				\begin{aligned}
					&  (d,\frac{2(d+3)}{d},\frac{2(d+3)}{d}) \quad if \; 1+\frac{6}{d}\le p<3 \\
					& (d,2(p-1),\frac{2d(p-1)}{3}) \quad if \; p\le 3.
				\end{aligned}
				\right.
			\end{equation*}
			From direct calculation, we see that the above ternary pairs are admissible and satisfy the second condition in Lemma \ref{today}.
		\end{proof}
		\begin{rem}\label{10122}
			In the focusing case $\mu=-1$, for any $p>1$, there is no asymptotic completeness. To be more concrete, we can construct the following soliton solution for the focusing DNLS, 
			\begin{equation*}
				u(x,t)=Q(x)e^{it\tau}, \tau>0,
			\end{equation*}
			where $Q$ satisfies the ground state equation $(\ref{ground})$. From Theorem 1.2 in \cite{7}, we see that the above ground state equation has a solution $Q\in \ell^{2}(\Z^{d})$, which implies the soliton solution $u\in C^{1}(\left[0,\infty\right);\ell^{2}(\Z^{d}))$ is a classical solution. Applying the discrete Fourier transform, we can immediately show that there is no asymptotic state $u_{+}\in \ell^{2}(\Z^{d})$ for the soliton solution $u$. Combined with Theorem \ref{6}, we directly show that, for $6\le(p-1)d$, the mass of the soliton solution $u$ and the $\ell^{2}(\Z^{d})$-norm of the ground state have a uniformly lower bound as follows
			\begin{equation*}
				\|u(\cdot,t)\|_{\ell^{2}(\Z^{d})}=\|Q\|_{\ell^{2}(\Z^{d})}\ge \varepsilon=\varepsilon(d)>0,
			\end{equation*}
			where the lower bound $\varepsilon$ is independent of nonlinearity order $p$. Particularly, it also serves as a complementary statement of \cite{35}, in which M. Weinstein has proved that for $(p-1)d<4$, there is no energy excitation threshold for every standing wave solution to the focusing DNLS (\ref{NLS}). To be more concrete, \cite{35} shows that for any $v>0$, there exists a ground state $Q$, such that $\|Q\|_{\ell^{2}(\Z^{d})}=v$. In conclusion, for the focusing case, it's necessary to pose an upper bound for the $\ell^{2}(\Z^{d})$-norm of the initial data $u_{0}$, ensuring the asymptotic completeness of the corresponding $u$. Naturally, there is a problem of finding the threshold of scattering for the focusing case, which was solved in the continuous background, see \cite{8}, but is still unsolved in the discrete background.
		\end{rem}
		\begin{rem}
			There still remains a tough problem of establishing the scattering theory for the solutions with large initial data in the defocusing case. The challenge stems from two inadequacies. The first is the absence of a strong decay estimate, like the Morawetz estimate in \cite{9,10}, to remove the requirement of norm boundedness (\ref{cc}). The second is the insufficiency of the distinction between the focusing case and the defocusing case. For instance, in continuous background, the positivity of conserved energy for the defocusing case can control the $H^{1}(\mathbb{R}^{d})$-norm of the global solution $u$, which surpasses the $L^{2}(\mathbb{R}^{d})$-norm control provided by the mass conservation. However, in discrete background, $H^{1}(\Z^{d})$-norm is equivalent to $\ell^{2}(\Z^{d})$-norm, meaning that the positivity of conserved energy fails to provide more control than conserved mass. 
		\end{rem}
		
		\section{Some extended results}
		We first establish the long-time well-posedness of the DNLS (\ref{small data}) with the small initial data, as a direct consequence of the continuation criterion given in Theorem \ref{1}.
		
		\begin{proof}[Proof of Theorem \ref{7}]
			For simplicity, we just deal with the case when $g^{jk}(u,\partial u)=g^{jk}(u)$ and $F(u,\partial u)=F(u)$, as the general case is parallel. Based on the continuation criterion, it suffices to show the total mass $M(t)=\|u(\cdot,t)\|_{\ell^{2}(\Z^{d})}$ is bounded in the time interval $\left[0,K\log(\log(\frac{1}{\varepsilon}))\right]$. 
			\par 
			Notice that, with the continuity method, if for a time interval $[0,T]$, there exists $R>0$ s.t. $M(0)<\frac{R}{2}$ and the property(P) that $\forall t\in [0,T]$, $M(t)\le R$ would imply $M(t)\le \frac{R}{2}$, then we can obtain $M(t)\le R, \forall t\in [0,T]$. Therefore, the energy $M(t)$ is bounded in $[0,T]$, which will ensure the existence of a classical solution in $[0,T]$. Thus, in the following, we choose $A$ satisfying $\|u_{0}\|_{\ell^{2}(\Z^{d})}\ll A$(to ensure $M(0)<\frac{A\varepsilon}{2}$) and  $R=A\varepsilon=\frac{1}{10}$ (this can be done if $\varepsilon$ is small enough). The whole theorem comes to show that for $T\lesssim \log(\log(\frac{1}{\varepsilon}))$, the property(P) is true on $[0,T]$.
			\par For convenience, we introduce the following constant 
			\begin{equation*}
				M:=\max \left\{\max_{j,k;|x|\le1}|g^{jk}(x)|,\max_{|x|\le1}|\nabla F(x)|, 1, d\right\}.
			\end{equation*}
			Then, from the hypothesis $M(t)\le A\varepsilon=\frac{1}{10}$, we immediately have the estimates below
			\begin{equation}\label{necessary}
				\|g^{jk}(u(\cdot,t))\|_{\ell^{\infty}(\Z^{d})}\lesssim M, \quad \|F(u(\cdot,t))\|_{\ell^{2}(\Z^{d})}\le M\|u(\cdot,t)\|_{\ell^{2}(\Z^{d})}=M\cdot M(t).
			\end{equation}
			Besides, applying the energy estimate in Theorem \ref{energy estimate2}, we obtain
			\begin{equation}\label{rr}
				M(t)\le C\left(M(0)+\int_{0}^{t}\|F(u(\cdot,t))\|_{\ell^{2}(\Z^{d})}ds\right)\cdot \exp\left(C'\int_{0}^{t}\sum_{j,k}\|g^{jk}(u(\cdot,s))\|_{\ell^{\infty}(\Z^{d})}ds\right).
			\end{equation}
			Combining (\ref{necessary}) and (\ref{rr}), we can further get
			\begin{equation*}
				M(t)\le Ce^{C'MT}\left(M(0)+\int_{0}^{t}M\cdot M(s)ds\right)\le CMe^{C'MT}\left(M(0)+\int_{0}^{t}M(s)ds\right).
			\end{equation*}
			We denote $F:=CMe^{C'MT}$, then the Gronwall inequality directly implies
			\begin{equation*}
				M(t)\le Fe^{Ft}M(0)\le  Fe^{FT}M(0).
			\end{equation*}
			Notice that we have the following equivalences
			\begin{equation*}
				Fe^{FT}M(0)\le \frac{A\varepsilon}{2}\Longleftrightarrow Fe^{FT}\lesssim \frac{A}{2}\Longleftrightarrow T\lesssim\log(\log(A))\approx\log(\log(\frac{1}{\varepsilon})).
			\end{equation*}
			Thus, we complete the proof of property(P), which ensures the existence of a classical solution in the time interval $\left[0,K\log(\log(\frac{1}{\varepsilon}))\right]$.
		\end{proof}
		If we consider the traditional case with $g^{jk}\equiv \frac{1}{2}\delta_{jk}$, then we can derive a better lower bound for the maximal existence time.
		\begin{thm}
			If $g^{jk}\equiv \frac{1}{2}\delta_{jk}$, then we have $T^{\ast}\ge K\log(\frac{1}{\varepsilon})$, where $K=K(F,u_{0},d)>0$.
		\end{thm}
		\begin{proof}
			We still follow the continuity method in the proof of Theorem \ref{7}, with the energy estimate in Theorem \ref{energy estimate1} instead of the energy estimate in Theorem \ref{energy estimate2}. Then we similarly choose $A$ such that $A\varepsilon=\frac{1}{10}$ to ensure the (\ref{necessary}) and can obtain 
			\begin{equation*}
				M(t)\le M\left(M(0)+\int_{0}^{t}M(s)ds\right).
			\end{equation*}
			Applying the Gronwall inequality, we can immediately derive the lower bound $T^{\ast}\ge K\log(\frac{1}{\varepsilon})$.
		\end{proof}
		Before we give the proof of Theorem \ref{8}, we shall introduce some useful lemmas.
		\begin{lemma}\label{product}
			For $a,b\in \left(0,1\right)\cup (1,+\infty)$, $1\ll t$, we have the following equivalence
			\begin{equation*}
				\int_{0}^{t}\frac{1}{\langle s \rangle^{a}}\frac{1}{\langle t-s \rangle^{b}}ds\approx\frac{1}{\langle t \rangle^{\min\left\{a,\; b,\; a+b-1\right\}}}.
			\end{equation*}
			\begin{proof}
				Splitting the above integral into the following three parts
				\begin{equation*}
					\int_{0}^{t}\frac{1}{\langle s \rangle^{a}}\frac{1}{\langle t-s \rangle^{b}}ds=\int_{0}^{1}+\int_{t-1}^{t}+\int_{1}^{t-1}:=(i)+(ii)+(iii).
				\end{equation*}
				For $(i)$ and $(ii)$, we can directly deduce that
				\begin{equation*}
					(i)\approx\int_{0}^{1}\frac{1}{\langle t-s \rangle^{b}}ds\approx\frac{1}{\langle t \rangle^{b}}, \quad (ii)\approx\int_{t-1}^{t}\frac{1}{\langle s \rangle^{a}}ds\approx\frac{1}{\langle t \rangle^{a}}.
				\end{equation*}
				For $(iii)$, we can change the variable $s:=t\xi$ and still split the integral into three parts
				\begin{equation*}
					(iii)\approx\frac{1}{t^{a+b-1}}\int_{\frac{1}{t}}^{\frac{t-1}{t}}\frac{1}{\xi^{a}}\frac{1}{(1-\xi)^{b}}d\xi=\frac{1}{t^{a+b-1}}\left(\int_{\frac{1}{t}}^{1/4}+\int_{\frac{1}{4}}^{\frac{3}{4}}+\int_{\frac{3}{4}}^{\frac{t-1}{t}}\right).
				\end{equation*}
				Similarly, we can derive the desired estimate.
			\end{proof}
		\end{lemma}
		Next, we follow the notation in the statement of Theorem \ref{8}, where $\sigma:=\frac{p-1}{2}$ and $q:=\frac{2d(2\sigma+1)}{d(2\sigma+1)+3}$.
		\begin{lemma}\label{1011}
			For $\frac{3}{d}\le\frac{\sigma(2\sigma+1)}{\sigma+1}$, $r\in \left[q',2(\sigma+1)\right]$, there exists $\varepsilon=\varepsilon(d)>0$, such that if the initial data $u_{0}$ satisfies the smallness condition $\|u_{0}\|_{\ell^{\frac{2(\sigma+1)}{2\sigma+1}}(\Z^{d})}\le \varepsilon$, then we have the following decay estimate for the corresponding solution $u$
			\begin{equation}\label{decay estimate2}
				\|u(\cdot,t)\|_{\ell^{r}(\Z^{d})}\lesssim\frac{1}{\langle t\rangle^{\frac{d(r-2)}{3r}}}\|u_{0}\|_{\ell^{r'}(\Z^{d})}.
			\end{equation}
		\end{lemma}
		\begin{proof}
			We first consider the special case $\frac{3}{d}<\frac{\sigma(2\sigma+1)}{\sigma+1}$ and $r\in \left(q',2(\sigma+1)\right]$. For fixed $r$, we define the norm $\|\cdot\|_{A}$ and the closed set $X$ for $u:\Z^{d}\times \left[0,+\infty\right)\to \mathbb{C}$ as follows
			\begin{equation*}
				\|u\|_{A}:=\|\langle t\rangle^{\frac{d(r-2)}{3r}} u\|_{L_{t}^{\infty}\ell_{x}^{r}(\left[0,+\infty\right)\times \Z^{d})},\quad X := \left\lbrace u \Big| \|u\|_{A}\le 2C \|u_{0}\|_{\ell^{r'}(\Z^{d})} \right\rbrace,
			\end{equation*}
			where $C$ is the constant from the following interpolation inequality for $v:\Z^{d}\to \mathbb{C}$
			\begin{equation}\label{decay}
				\|e^{it\Delta/2}v\|_{\ell^{r}(\Z^{d})}\le \frac{C}{\langle t\rangle^{\frac{d(r-2)}{3r}}}\|v\|_{\ell^{r'}(\Z^{d})}, \forall r\ge2.
			\end{equation}
			Next, we define the map $T$ as follows
			\begin{equation*}
				Tu:= e^{it\Delta/2}u_{0}\pm i\int_{0}^{t}e^{i(t-s)\Delta/2}|u(\cdot,s)|^{2\sigma}u(\cdot,s)ds.
			\end{equation*}
			To apply the contraction mapping principle, we just need to check that $T$ maps $X$ to $X$, and $T$ has the contractivity. For $\varepsilon$ sufficiently small and $u\in X$, we can apply the above interpolation inequality (\ref{decay}), Lemma \ref{product} and the Minkowski inequality to deduce that
			\begin{equation*}
				\|Tu(\cdot,t)\|_{\ell^{r}(\Z^{d})}\le C\|u_{0}\|_{\ell^{r'}(\Z^{d})}\frac{1}{\langle t\rangle^{\frac{d(r-2)}{3r}}}+C\int_{0}^{t}\|u(\cdot,s)\|_{\ell^{(2\sigma+1)r'}}^{2\sigma+1}\frac{1}{\langle t-s\rangle^{\frac{d(r-2)}{3r}}}ds
			\end{equation*}
			\begin{equation*}
				\le C\|u_{0}\|_{\ell^{r'}(\Z^{d})}\frac{1}{\langle t\rangle^{\frac{d(r-2)}{3r}}}+C(2C)^{2\sigma+1}\|u_{0}\|_{\ell^{r'}(\Z^{d})}^{2\sigma+1}\int_{0}^{t}\frac{1}{\langle t-s\rangle^{\frac{d(r-2)}{3r}}}\frac{1}{\langle s\rangle^{\frac{(2\sigma+1)d(r-2)}{3r}}}ds
			\end{equation*}
			\begin{equation*}
				\le 2C\|u_{0}\|_{\ell^{r'}(\Z^{d})}\frac{1}{\langle t\rangle^{\frac{d(r-2)}{3r}}}.
			\end{equation*}
			Note that the condition $r\in \left(q',2(\sigma+1)\right]$ ensures $(2\sigma+1)r'\ge r,\; \frac{(2\sigma+1)d(p-2)}{3p}>1$. Thus, it remains to show the contractivity of $T$. For $u,v\in X$,
			\begin{equation*}
				\|Tu-Tv\|_{A}\le C\int_{0}^{t}\||u(\cdot,s)|^{2\sigma}u(\cdot,s)-|v(\cdot,s)|^{2\sigma}v(\cdot,s)\|_{\ell^{r'}(\Z^{d})}\frac{\langle t\rangle^{\frac{d(r-2)}{3r}}}{\langle t-s\rangle^{\frac{d(r-2)}{3r}}}ds
			\end{equation*}
			\begin{equation*}
				\lesssim \int_{0}^{t}\|u(\cdot,s)-v(\cdot,s)\|_{\ell^{r}(\Z^{d})}\left(\|u(\cdot,s)\|_{\ell^{r}(\Z^{d})}^{2\sigma}+\|v(\cdot,s)\|_{\ell^{r}(\Z^{d})}^{2\sigma}\right)\frac{\langle t\rangle^{\frac{d(r-2)}{3r}}}{\langle t-s\rangle^{\frac{d(r-2)}{3r}}}ds
			\end{equation*}
			\begin{equation*}
				\lesssim \|u_{0}\|_{\ell^{r'}(\Z^{d})}^{2\sigma}\cdot\|u-v\|_{A}\le \frac{1}{2}\|u-v\|_{A}.
			\end{equation*}
			For the general case, we note that the implicit constant in the decay estimate (\ref{decay estimate2}) is independent of $r$. Taking the limit $r\to q'$, we can immediately conclude the whole theorem. 
		\end{proof}
		\begin{rem}\label{1012}
			This proof is inspired by \cite{5} and partially answers the M. Weinstein's conjecture in \cite{35}. To be more specific, Lemma \ref{1011} shows that, under some smallness condition, the solution of the DNLS will disperse to zero, in the sense that for $r\ge q'>2$, 
			\begin{equation*}
				\|u(\cdot,t)\|_{\ell^{r}(\Z^{d})}\to 0, \; t\to +\infty.
			\end{equation*}
			Additionally, the conditions $r\in \left[q',2(\sigma+1)\right]$ and $\|u_{0}\|_{\ell^{\frac{2(\sigma+1)}{2\sigma+1}}(\Z^{d})}\le \varepsilon$ in Lemma \ref{1011} can be respectively replaced by $r\in \left[q',\tau\right]$ and $\|u_{0}\|_{\ell^{\tau'}(\Z^{d})}\le \varepsilon$, for any $\tau\in \left[q',2(\sigma+1)\right]$. In particular, we can take $\tau=q'$, which is crucial in the proof of Theorem \ref{8}. 
		\end{rem}
		Now we can immediately prove Theorem \ref{8}.
		\begin{proof}[Proof of Theorme \ref{8}]
			Combined with the analysis in Remark \ref{1012} and Remark \ref{10122}, we can directly summarize the uniformly lower bound property of the ground state as the three tables in Theorem \ref{8}.
		\end{proof}

		\newpage
		\section*{Acknowledgement}
		The author is supported by NSFC, No. 123B1035 and is grateful to Prof. B. Hua for helpful suggestions.
		
		\section*{Conflict of interest statement}
		The author does not have any possible conflict of interest.
		
		\section*{Data availability statement}
		The manuscript has no associated data.
		\bigskip
		\bigskip

		\bibliographystyle{alpha}
		\bibliography{second}
		
	\end{document}